\documentclass[12pt,a4paper,twoside]{amsart}
\pdfoutput=1
\usepackage[english]{babel}
\usepackage{lmodern}
\usepackage{microtype}
\usepackage{mathtools}
\usepackage{esint}
\usepackage{comment}
\usepackage[shortlabels]{enumitem}
\usepackage{hyperref}
\usepackage{amssymb}

\title[Inclusions for conductivity equations]{Superconductive and insulating inclusions for linear and non-linear conductivity equations}
\author{Tommi Brander}
\address{University of Jyv\"askyl\"a, Department of Mathematics and Statistics.}
\email{tommi.o.brander@jyu.fi}
\author{Joonas Ilmavirta}
\address{University of Jyv\"askyl\"a, Department of Mathematics and Statistics.}
\email{joonas.ilmavirta@jyu.fi}
\author{Manas Kar}
\address{University of Jyv\"askyl\"a, Department of Mathematics and Statistics.}
\email{manas.m.kar@jyu.fi}

\newcommand*{\R}{\mathbb{R}}

\newcommand*{\N}{\mathbb{N}}

\newcommand*{\eps}{\varepsilon}
\newcommand*{\doo}{\partial}
\newcommand*{\ol}[1]{\overline{#1}}
\newcommand{\sulut}[1]{\left( #1 \right)}
\newcommand{\joukko}[1]{\left\{ #1 \right\}}
\newcommand{\abs}[1]{\left\lvert #1 \right\rvert}
\newcommand{\aabs}[1]{\left\| #1 \right\|}
\newcommand{\norm}[1]{\left\| #1 \right\|}
\newcommand{\der}{\mathrm{d}}

\newcommand{\Der}[1]{\frac{\der}{\der #1}}
\newcommand{\ip}[2]{\left\langle#1,#2\right\rangle}

\newcommand{\sphere}{\mathbb{S}}

\theoremstyle{plain}
\newtheorem{theorem}{Theorem}
\newtheorem{lemma}[theorem]{Lemma}
\newtheorem{corollary}[theorem]{Corollary}
\newtheorem{proposition}[theorem]{Proposition}

\theoremstyle{definition}
\newtheorem{definition}[theorem]{Definition}

\theoremstyle{remark}
\newtheorem{remark}[theorem]{Remark}

\DeclareMathOperator{\dive}{div}

\DeclareMathOperator{\dist}{dist}

\DeclareMathOperator{\Tr}{Tr}
\DeclareMathOperator{\expo}{exp}

\DeclareMathOperator{\diam}{diam}

\def\vint_#1{\mathchoice%
          {\mathop{\kern 0.2em\vrule width 0.6em height 0.69678ex
depth -0.58065ex
                  \kern -0.8em \intop}\nolimits_{\kern -0.4em#1}}%
          {\mathop{\kern 0.1em\vrule width 0.5em height 0.69678ex
depth -0.60387ex
                  \kern -0.6em \intop}\nolimits_{#1}}%
          {\mathop{\kern 0.1em\vrule width 0.5em height 0.69678ex
              depth -0.60387ex
                  \kern -0.6em \intop}\nolimits_{#1}}%
          {\mathop{\kern 0.1em\vrule width 0.5em height 0.69678ex
depth -0.60387ex
                  \kern -0.6em \intop}\nolimits_{#1}}}
\def\vintslides_#1{\mathchoice%
          {\mathop{\kern 0.1em\vrule width 0.5em height 0.697ex depth -0.581ex
                  \kern -0.6em \intop}\nolimits_{\kern -0.4em#1}}%
          {\mathop{\kern 0.1em\vrule width 0.3em height 0.697ex depth -0.604ex
                  \kern -0.4em \intop}\nolimits_{#1}}%
          {\mathop{\kern 0.1em\vrule width 0.3em height 0.697ex depth -0.604ex
                  \kern -0.4em \intop}\nolimits_{#1}}%
          {\mathop{\kern 0.1em\vrule width 0.3em height 0.697ex depth -0.604ex
                  \kern -0.4em \intop}\nolimits_{#1}}}

\newcommand{\aveint}[2]{\mathchoice%
          {\mathop{\kern 0.2em\vrule width 0.6em height 0.69678ex
depth -0.58065ex
                  \kern -0.8em \intop}\nolimits_{\kern -0.45em#1}^{#2}}%
          {\mathop{\kern 0.1em\vrule width 0.5em height 0.69678ex
depth -0.60387ex
                  \kern -0.6em \intop}\nolimits_{#1}^{#2}}%
          {\mathop{\kern 0.1em\vrule width 0.5em height 0.69678ex
depth -0.60387ex
                  \kern -0.6em \intop}\nolimits_{#1}^{#2}}%
          {\mathop{\kern 0.1em\vrule width 0.5em height 0.69678ex
depth -0.60387ex
                  \kern -0.6em \intop}\nolimits_{#1}^{#2}}}

\numberwithin{theorem}{section}
\numberwithin{equation}{section}

\begin{document}

\begin{abstract}
We detect an inclusion with infinite conductivity from boundary measurements represented by the Dirichlet-to-Neumann map for the conductivity equation.
We use both the enclosure method and the probe method.
We use the enclosure method to prove partial results when the underlying equation is the quasilinear $p$-Laplace equation.
Further, we rigorously treat the forward problem for the partial differential equation $\operatorname{div}(\sigma\lvert\nabla u\rvert^{p-2}\nabla u)=0$ where the measurable conductivity $\sigma\colon\Omega\to[0,\infty]$ is zero or infinity in large sets and $1<p<\infty$.
\end{abstract}



\maketitle


\section{Introduction}
\label{sec:intro}

We study inverse boundary value problems for the partial differential equation
\begin{equation}
\label{eq:pde}
\dive(\sigma(x)\abs{\nabla u(x)}^{p-2}\nabla u(x))=0,
\end{equation}
where the measurable coefficient $\sigma\geq0$ is allowed to take the values~$0$ and~$\infty$ in large sets and the exponent is in the range $1<p<\infty$.
This includes the case $p=2$ where our PDE becomes the linear conductivity equation $\dive(\sigma\nabla u)=0$ which appears in Calder\'on's inverse problem~\cite{Calderon:1980}.

To arrive at the PDE from a physical starting point, we consider an electric potential (voltage) $u\colon\Omega \to\R$, where $\Omega \subset \R^n$.
By Ohm's law current density is given by $J(x)=-\sigma(x)\nabla u(x)$ and Kirchhoff's law entails $\dive(J(x))=0$.
To arrive at the non-linear equation~\eqref{eq:pde} instead of $\dive(\sigma\nabla u)=0$, we replace Ohm's law with the non-linear law $J(x)=-\sigma(x)\abs{\nabla u(x)}^{p-2}\nabla u(x)$.
Physically such non-linear laws can occur in dielectrics~\cite{Bueno:Longo:Varela:2008,Garroni:Kohn:2003,Garroni:Nesi:Ponsiglione:2001,Kohn:Levy:1998,Talbot:Willis:1994:a,Talbot:Willis:1994:b}, 
plastic moulding~\cite{Aronsson:1996,King:Richardson:2015}, 
electro-rheological and thermo-rheological fluids~\cite{Antontsev:Rodrigues:2006,Berselli:Diening:Ruzicka:2010,Ruzicka:2000}, 
viscous flows in glaciology~\cite{Glowinski:Rappaz:2003} and plasticity phenomena~\cite{Atkinson:Champion:1984,Idiart:2008,PonteCastaneda:Suquet:1998,PonteCastaneda:Willis:1985,Suquet:1993}.

Kirchhoff's law retains its linear form and consequently our PDE is of divergence form.
This is convenient for the study of weak solutions and calculus of variations.
Weak solutions of~\eqref{eq:pde} in a domain $\Omega\subset\R^n$ are minimizers of the energy functional
\begin{equation}
\label{eq:energy}
E(u)
=
\int_\Omega\sigma(x)\abs{\nabla u(x)}^p\der x
\end{equation}
in the space~$W^{1,p}(\Omega)$ with some prescribed boundary values.
This is true even if~$\sigma$ takes the values~$0$ and~$\infty$ in non-empty open sets.
This is our main result for the direct problem, and the details will be discussed in sections~\ref{sec:direct-intro} and~\ref{sec:direct}.

Our PDE can be classified as a quasilinear elliptic equation.
We do not, however, bound the coefficient~$\sigma$ away from zero or infinity, so ellipticity holds in a weaker sense than usual.
If $\sigma\equiv1$, the solutions of~\eqref{eq:pde} are known as $p$-harmonic functions.
They have been studied extensively (see for instance~\cite{Ladyzhenskaya:Ural'tseva:1968,Gilbarg:Trudinger:1983,Heinonen:Kilpelainen:Martio:1993,Lindqvist:2006} and the references therein), but inverse problems for elliptic equations of this type have received considerably less attention.

We assume our potential to be real-valued since this is physically most relevant.
For complex-valued functions~$u$ one can obtain essentially the same results with the same tools, but we restrict our attention to the real case (apart from section~\ref{sec:p2-enclosure}).

Our goal is, given Dirichlet and Neumann boundary values of all solutions of~\eqref{eq:pde}, to reconstruct the shape of an unknown obstacle having zero or infinite conductivity.
This data is encoded in the so-called Dirichlet-to-Neumann map which we will describe in section~\ref{sec:direct-intro} and in more detail in section~\ref{sec:dn-map}.
When $p\neq2$, there are very few results in this direction.
When $p=2$, this is a version of Calder\'on's famous inverse boundary value problem.
We will summarize earlier results in section~\ref{sec:old-results} and new results proven in this article in section~\ref{sec:new-results}.

\subsection{The direct problem}
\label{sec:direct-intro}

Before embarking on a study of inverse problems, it is good to show that the direct problem is well-posed.
The well-posedness result we present is, to the best of our knowledge, new in its generality.

Let~$\Omega$ be a bounded domain and $\sigma\colon\Omega\to[0,\infty]$ a measurable function.
For simplicity, we make the standing assumption that the sets $D_0=\sigma^{-1}(0)$ and $D_\infty=\sigma^{-1}(\infty)$ are open and the three sets~$\partial\Omega$,~$\bar D_0$ and~$\bar D_\infty$ are disjoint.
We also assume that outside the sets~$D_0$ and~$D_\infty$ the function~$\sigma$ is bounded away from both zero and infinity.

In this setting we look for minimizers of the energy~\eqref{eq:energy} in~$W^{1,p}(\Omega)$, given boundary values $f\in W^{1,p}(\Omega)/W^{1,p}_0(\Omega)$.
Minimization of this energy corresponds to solving an Euler--Lagrange equation.

Usually such problems are posed on the domain $\Omega\setminus(\bar D_0\cup\bar D_\infty)$ and one would impose suitable boundary conditions on~$\partial D_0$ and~$\partial D_\infty$.
We prefer to work with all of~$\Omega$ when possible, for in the inverse problem the function~$\sigma$ and therefore the sets~$D_0$ and~$D_\infty$ are unknown.
See remark~\ref{rmk:pde-bdy} for the usual formulation.

\begin{theorem}
\label{thm:direct-summary}
Let $\Omega\subset\R^n$, $n\geq1$, be a bounded open set.
Let $\sigma\colon\Omega\to[0,\infty]$ be a measurable function, and denote $D_0=\sigma^{-1}(0)$ and $D_\infty=\sigma^{-1}(\infty)$.
Assume that both~$D_0$ and~$D_\infty$ are open and Lipschitz, and the three sets~$\partial\Omega$, $\bar D_0$ and~$\bar D_\infty$ are disjoint.
Assume furthermore that~$\sigma$ is bounded away from zero and infinity outside the sets~$D_0$ and~$D_\infty$.

Let $p\in(1,\infty)$.
Fix any boundary value $f\in W^{1,p}(\Omega)/W^{1,p}_0(\Omega)$.
The energy~\eqref{eq:energy} has a minimizer $u\in W^{1,p}(\Omega)$ with boundary values $u|_{\partial\Omega}=f$.
The minimizer is unique up to functions that have vanishing gradient outside~$D_0$ and that vanish on~$\doo \Omega$.

A function $u\in W^{1,p}(\Omega)$ is such a minimizer if and only if it satisfies the Euler--Lagrange equation $\dive(\sigma\abs{\nabla u}^{p-2}\nabla u)=0$ weakly in the sense that
\begin{equation}
\int_\Omega\sigma\abs{\nabla u}^{p-2}\nabla u\cdot\nabla\phi \der x
= 0
\end{equation}
for all $\phi\in W^{1,p}_0(\Omega)$ with $\nabla \phi = 0$ in $D_\infty$.
\end{theorem}

For a proof, see theorems~\ref{thm:direct-variation} and~\ref{thm:direct-pde} and their proofs.

This theorem is true, in particular, in the linear case $p=2$.
Then the PDE is $\dive(\sigma\nabla u)=0$.

We point out that even though the minimizer~$u$ of the energy is not unique, the vector-valued function $\sigma\abs{\nabla u}^{p-2}\nabla u$ is unique.
Since the exceptional sets~$D_0$ and~$D_\infty$ do not reach the boundary, the boundary values can be interpreted in the Sobolev sense as usual.

The information about boundary values of solutions is encoded in the Dirichlet-to-Neumann map (DN map)~$\Lambda_\sigma$ from the quotient space $W^{1,p}(\Omega)/W^{1,p}_0(\Omega)$ to its dual.
Let $f,g\in W^{1,p}(\Omega)/W^{1,p}_0(\Omega)$ be any functions and let $\bar f,\bar g\in W^{1,p}(\Omega)$ be their extensions so that $\dive(\sigma\abs{\nabla\bar f}^{p-2}\nabla\bar f)=0$, while the extension~$\bar g$ satisfies~$\nabla \bar g = 0$ in $D_\infty$.
We define~$\Lambda_\sigma$ so that
\begin{equation}\label{eq:dn-map}
\ip{\Lambda_\sigma f}{g}
=
\int_\Omega\sigma\abs{\nabla\bar f}^{p-2}\nabla\bar f\cdot\nabla\bar g.
\end{equation}
The Dirichlet-to-Neumann map is linear if and only if $p=2$.

\subsection{Known results}
\label{sec:old-results}

Let us first fix $p=2$.
If~$\sigma$ is assumed to be bounded away from zero and infinity, the DN map~$\Lambda_\sigma$ determines~$\sigma$ in two dimensions~\cite{Astala:Paivarinta:2006}.
In higher dimensions this is true if~$\sigma$ is additionally assumed to be Lipschitz~\cite{Caro:Rogers:2016}.
For more information about Calder\'on's problem, we refer to~\cite{Calderon:1980,Uhlmann:2009}.
Astala, Lassas and P\"aiv\"arinta~\cite{Astala:Lassas:Paivarinta:2016} have investigated anisotropic conductivities that are not bounded from above or away from zero.

A different kind of problem is to reconstruct an inclusion --- a subdomain of~$\Omega$ with zero, infinite or non-zero finite conductivity in a constant background conductivity --- from the DN map.
Problems of this kind are our object of study.

There are several methods for recovering inclusions with zero or finite conductivity by making boundary measurements of both voltages and currents~\cite{Isakov:1988,Colton:Kirsch:1996,Cakoni:Colton:2006,Kirsch:Grinberg:2008,Harrach:2013,Ikehata:1998,Potthast:2005,Ikehata:1999:jan,Nakamura:Uhlmann:Wang:2005,Harrach:Ullrich:2013}.
For regions of infinite conductivity a recent article by Ramdani and Munnier~\cite{Munnier:Ramdani:2017} shows how to detect infinitely conductive bodies from the Dirichlet to Neumann map in two dimensional domain for the linear conductivity equation.
Their method is based on geometry in the complex plane, in particular Riemann mappings.
Br\"uhl~\cite[section 4.3.1]{Bruhl:1999} and Schmitt~\cite[section 2.2.2]{Schmitt:2010} have used the factorization method to detect perfectly conducting inclusions.
Alessandrini and Valenzuela~\cite{Alessandrini:Valenzuala:1996} detect perfectly conducting or insulating cracks with two bondary measurements.
Friedman and Vogelius~\cite{Friedman:Vogelius:1989} have shown that one can recover the location and scale of a finite number of small inclusions with zero or infinite conductivity in an inhomogeneous background from the DN map.
Superconductive but grounded inclusions have been detected in, for example,~\cite{Kress:Rundell:2005,Borman:Inghman:Johansson:Lesnic:2009}.
Moradifam, Nachman and Tamasan~\cite{Moradifam:Nachman:Tamasan:2012} consider a single interior measurement for conductivity equation and detect insulating or perfectly conducting inclusions, though their approach to the direct problem is not variational.
Perfectly conducting inclusions in the context of the Maxwell equations have been detected by sampling methods~\cite{Gebauer:Hank:Kirsch:Muniz:Schneider:2005,Cakoni:Fares:Haddar:2006}.

Other results relating to infinitely conducting obstacles concern the situation of two such obstacles being close to each other and the main concern is the blow-up of the solutions~\cite{Kang:Lim:Yun:2013,Gorb:Novikov:2012}.

For other values of~$p$ much less is known.
We assume $1<p<\infty$ throughout this article.
The $p$-conductivity equation~\eqref{eq:pde} with infinite conductivity has been considered by Gorb and Novikov~\cite{Gorb:Novikov:2012} for $2 \leq p \in \N$ in dimensions two and three.
We are not aware of a rigorous treatment of the forward problem in the literature, so we provide one in section~\ref{sec:direct} as summarized above in section~\ref{sec:direct-intro}.

The $p$-Calder\'on's problem, or Calder\'on's problem related to the $p$-conductivity equation, was introduced by Salo and Zhong~\cite{Salo:Zhong:2012}.
They recover the conductivity on the boundary of the domain.
Brander~\cite{Brander:2016:jan} improved the result to first order derivative of conductivity on the boundary, but with increased regularity assumptions.
A recent result by Brander, Kar and Salo~\cite{Brander:Kar:Salo:2015} shows that one can detect the convex hull of an inclusion with conductivity bounded away from zero and infinity. Further results related to the enclosure method and the monotonicity method for the $p$-Laplace equation can be found in \cite{Brander:Harrach:Kar:Salo:2017}.
Very recently Guo, Kar and Salo~\cite{Guo:Kar:Salo:2016} proved that under a monotonicity assumption the DN~map is injective for Lipschitz conductivites when~$n=2$ for general~$p$, and when~$n \geq 3$ when one of the conductivities is almost constant.
Their results assume the conductivity to be bounded from above and away from zero.

\subsection{New results}
\label{sec:new-results}

First, we study the problem of recovering the set~$D_\infty$ from the DN map when $p=2$, $D_0=\emptyset$ and $\sigma\equiv1$ outside~$D_\infty$.
Here we mainly follow the probe and enclosure methods of Ikehata~\cite{Ikehata:1998, Ikehata:1999:jan} to reconstruct the unknown inclusion. 

\begin{theorem}
\label{thm:p2-inverse}
Let $\Omega\subset\R^n$, $n\geq2$, be a bounded~$C^1$ domain and let~$D$ be a Lipschitz domain compactly contained in~$\Omega$.
Assume furthermore that each connected component of $\bar\Omega\setminus D$ meets~$\partial\Omega$.
Let
\begin{equation}
\sigma(x)
=
\begin{cases}
1 & \text{ when }x\notin D\\
\infty & \text{ when }x\in D.
\end{cases}
\end{equation}
If $p=2$, then the Dirichlet-to-Neumann map determines the set~$D$.
\end{theorem}

In the notation introduced for the direct problem, $D=D_\infty$.
We will prove this theorem using the probe method.
When we use the enclosure method instead, we can prove recoverability of the convex hull of~$D_\infty$ if $D_\infty\subset\R^n$ is connected.
For a precise statement of this result, see proposition~\ref{prop:p2-enclosure}.
For more details and proofs, see section~\ref{sec:linear}.

For general $p\in(1,\infty)$ we only use the enclosure method.
The Runge approximation property for the $p$-Laplace equation is not known, so the probe method is not applicable for this non-linear model. 
We assume that one of~$D_0$ and~$D_\infty$ is empty and the other domain (called~$D$) has Lipschitz boundary.
We are unable to detect the convex hull of~$D$, but we can recover a larger set, giving an estimate from above for the inclusion.
We can also determine whether $D=\emptyset$ or not.
See corollary~\ref{cor:general-p} and the preceeding discussion for more details.

If we allow~$\sigma$ to take the values~$0$ or~$\infty$ in large sets, the DN map~$\Lambda_\sigma$ no longer determines~$\sigma$ uniquely.
For example, consider the domain $\Omega=B(0,3)\subset\R^n$ and a function $\sigma\colon\Omega\to[0,\infty]$ which is zero or infinity on $B(0,2)\setminus\bar B(0,1)$.
Then the values of~$\sigma$ in~$B(0,1)$ have no effect on the DN map.
We can only ever hope to recover~$\sigma$ up to the ``outermost boundaries of~$D_0$ and~$D_\infty$'' and whether or not these boundaries belong to~$D_0$ or~$D_\infty$.

\subsection*{Acknowledgements}
We would like to thank professor Eric Bonnetier for letting us know of the paper of Kang, Lim and Yun~\cite{Kang:Lim:Yun:2013} and thereby the paper of Gorb and Novikov~\cite{Gorb:Novikov:2012}.
We would also like to thank professor Mikko Salo for several discussions and the anonymous referees for their insightful comments and suggestions.
Part of the work was done during a visit to Institut Henri Poincar\'e, Paris, with financial support from the institute.
J.I.\ and M.K.\ were partially supported by an ERC Starting Grant (grant agreement no 307023).
T.B.\ was partially supported by the Academy of Finland through the Finnish Centre of Excellence in Inverse Problems Research.

\section{Detecting perfectly conducting inclusions for $p=2$}
\label{sec:linear}

In this section we consider the linear case $p=2$.
Our problem is to detect an inclusion of superconductive material within a homogeneous background medium.
We justify two types of reconstruction methods: the probe method and the enclosure method.
These methods were proposed by Ikehata.
In~\cite{Ikehata:1999:jan}, he introduced the enclosure method where he first used the CGO solutions with linear phase to detect the convex hull of the obstacle with finite or zero conductivity.
Regarding the probe method~\cite{Ikehata:1998}, he proposed to use the fundamental solution as a test function to detect the unknown inclusions of zero and finite conductivity for the conductivity equation.
Since then a lot of work has been done in this direction for various linear PDEs, see~\cite{Ikehata:1999:jan,Ikehata:2010} for an overview.
We mention some examples: the Helmholtz model~\cite{Nagayasu:Uhlmann:Wang:2011,Nakamura:Yoshida:2007}, Maxwell systems~\cite{Kar:Sini:2014:jan,Zhou:2010} and the linear elasticity equations~\cite{Kar:Sini:2014,Nakamura:Uhlmann:Wang:2005}.

In this article we apply these two methods to the linear conductivity model to detect the perfectly conductive obstacles and this is the first result which concerns the perfectly conductive case.
In the zero conductivity case, estimates for the difference of the DN maps do not involve any boundary integrals~\cite[Lemma 4.1]{Ikehata:1999:jan}.
However, due to the presence of the boundary integral, see equation~\eqref{eq:estimate-i}, it is more difficult to obtain the right kind of estimate in our situation.
We overcome the technical difficulties by using the boundary layer potential theory for the Laplace operator.
The rigorous analysis is in subsection~\ref{boundryEstimaTE}.
  
Let us repeat the setting of theorem~\ref{thm:p2-inverse}.
Let $\Omega \subset \R^n$, $n \geq 2$, be a bounded domain with~$C^1$ boundary and $D\subset\Omega$ be an open Lipschitz domain so that all components of~$\Omega\setminus {D}$ meet~$\partial\Omega$.
If~$\partial\Omega$ is connected, this amounts to requiring that~$\Omega\setminus {D}$ is connected.
We consider the obstacle problem
\begin{equation}
\label{linear_1}
\begin{cases}
\Delta u = 0 \text{ in } \Omega \setminus \bar{D} \\
u = \text{constant in each component of } D \\
\int_{\doo D} \frac{\doo u}{\doo\nu} \der S = 0  \\
u = f \text{ on } \doo \Omega,
\end{cases}
\end{equation}
where~$\nu$ is the outward unit normal to~$\Omega \setminus \bar{D}$.
The above problem formulates the situation where~$D$ is superconductive.
To see why this formulation corresponds to $\sigma\equiv\infty$ in~$D$, how to formulate the problem in~$\Omega$ instead of~$\Omega\setminus\bar D$ and how to pose the problem in the weak form, see section~\ref{sec:direct} and especially remark~\ref{rmk:pde-bdy}.
Notice that the constant value of~$u$ in the connected components depends on~$f$ and can be different in different components.

We index the Dirichlet to Neumann map by the domain~$D$ instead of the function~$\sigma$.
The DN map $\Lambda_D \colon H^{1/2}(\doo\Omega) \to H^{-1/2}(\doo\Omega)$ is defined by $\Lambda_D(f) = \partial_\nu u|_{\doo\Omega}$.
A weak version of~$\Lambda_D$ is given by
\begin{equation}
\ip{\Lambda_D f}{g} = \int_{\Omega\setminus\bar{D}} \nabla u \cdot \nabla\phi \der x
\end{equation}
where $g\in H^{1/2}(\doo\Omega)$, $\phi\in H^1(\Omega)$ with $\phi|_{\doo\Omega} = g$, $\nabla \phi = 0$ in $D$, and $u$ satisfies~\eqref{linear_1}.

We define similarly a DN map in the case $D=\emptyset$; this is just the DN map for the Laplacian on~$\Omega$.
We call it the free DN map and denote it by~$\Lambda_\emptyset$.
That is, $\Lambda_\emptyset(f) = \partial_\nu u_0|_{\doo\Omega}$ where~$u_0$ solves
\begin{equation}\label{linear_2}
\begin{cases}
\Delta u_0 = 0 \ \text{in}\ \Omega \\
u_0 = f \ \text{on}\ \doo\Omega
\end{cases}
\end{equation}
in the weak sense.

Our goal is to reconstruct the shape of the superconductive inclusion~$D$ from the knowledge of the Dirichlet-to-Neumann map~$\Lambda_D$ measured at~$\doo\Omega$.

Before presenting the enclosure and the probe method, we would like to state the following inequalities.
\subsection{Integral inequalities}
Let $w=u-u_0$ be the reflected solution satisfying
\begin{equation}\label{refl_lin}
\begin{cases}
\Delta w = 0 \ \text{in}\ \Omega\setminus\ol{D} \\
w+u_0\text{ is constant in each component of }D \\
\int_{\doo D}\frac{\doo w}{\doo\nu} \der S = -\int_{\doo D}\frac{\doo u_0}{\doo\nu} \der S \\
w= 0 \ \text{on}\ \partial \Omega
\end{cases}
\end{equation}
where~$u$ and~$u_0$ are the solutions of~\eqref{linear_1} and~\eqref{linear_2} respectively.

\begin{lemma}\label{linLem1}
For any $f\in H^{1/2}(\doo\Omega)$, we have the inequalities
\begin{equation}
\label{eq:estimate-i}
\begin{split}
\langle (\Lambda_D-\Lambda_{\emptyset})f, f\rangle
&\leq
\int_{D}\abs{\nabla u_0}^2\der x - 2 \int_{\doo D}\frac{\doo w}{\doo\nu} \ u_0 \ \der S
\end{split}
\end{equation}
and
\begin{equation}
\label{eq:estimate-ii}
\int_{D}\abs{\nabla u_0}^2\der x  \leq \langle (\Lambda_D-\Lambda_{\emptyset})f, f\rangle.
\end{equation}
\end{lemma}

\begin{proof}
\textbf{Proof of~\eqref{eq:estimate-i}:} Note that, 
\begin{equation}
\langle \Lambda_D f, g \rangle = \int_{\Omega\setminus\ol{D}} \nabla u \cdot \nabla\phi \der x,
\end{equation}
where $g\in H^{1/2}(\doo\Omega), \phi\in H^1(\Omega)$ with $\phi|_{\doo\Omega} = g$, $\nabla \phi = 0$ in $D$, and~$u$ satisfies~\eqref{linear_1}.
Since $u=u_0=f$ on~$\doo\Omega$, and by taking $\phi =u$, the left hand side of the above integral identity becomes
\begin{equation}
\ip{\Lambda_D f}{f}
= \int_{\Omega\setminus\ol{D}} \abs{\nabla u}^2\der x.
\end{equation}
Similarly, the free DN map can be written as
\begin{equation}
\ip{\Lambda_{\emptyset} f}{g} = \int_{\Omega} \nabla u_0 \cdot \nabla \phi \der x,
\end{equation}
where $g\in H^{1/2}(\doo\Omega)$ and $\phi \in H^1(\Omega)$ with $g=\phi|_{\doo\Omega}$ and~$u_0$ satisfies~\eqref{linear_2}.
Since $u=u_0=f$ on~$\doo\Omega$, replacing $\phi=u$ or $\phi=u_0$ in the above inequality gives us
\begin{equation}
\label{free}
\ip{\Lambda_{\emptyset} f}{f}
=\int_{\Omega}\abs{\nabla u_0}^2\der x
= \int_{\Omega}\nabla u_0 \cdot \nabla u \der x
\end{equation}
Therefore,
\begin{equation}
\begin{split}
&
\langle (\Lambda_D - \Lambda_{\emptyset}) f, f \rangle
\\\quad
&=
\int_{\Omega\setminus\ol{D}} \abs{\nabla u}^2\der x - \int_{\Omega}\abs{\nabla u_0}^2\der x
\\\quad
& = \left(\int_{\Omega\setminus\ol{D}} \abs{\nabla u}^2 \der x - \int_{\Omega\setminus\ol{D}} \abs{\nabla u_0}^2 \der x\right)
- \int_D \abs{\nabla u_0}^2 \der x.
\end{split}
\end{equation}
Using the inequality
\begin{equation}
\abs{\eta}^2-\abs{\zeta}^2 \geq 2\zeta\cdot(\eta-\zeta) \text{ for }\eta, \zeta \in \R^n
\end{equation}
we have
\begin{equation}\label{ineA}
\begin{split}
\langle (\Lambda_D - \Lambda_{\emptyset}) f, f \rangle
&\geq
2\int_{\Omega\setminus\ol{D}}\nabla u_0\cdot\nabla(u-u_0)\der x - \int_{D}\abs{\nabla u_0}^2\der x.
\end{split}
\end{equation}
Recall that, $w=u-u_0$ satisfies~\eqref{refl_lin}.
So, multiplying by~$u_0$ on both sides of $\Delta w =0$ and integrating by parts we obtain
\begin{equation}
-\int_{\Omega\setminus\ol{D}} \nabla w \cdot\nabla u_0 + \int_{\doo\Omega \cup \doo D}\frac{\doo w}{\doo\nu} \ u_0 \ \der S =0,
\end{equation}
that is,
\begin{equation}\label{ineB}
\begin{split}
\ip{(\Lambda_D - \Lambda_{\emptyset}) f}{f}
&=
\int_{\partial\Omega}\frac{\doo w}{\doo\nu}u_0\der S
\\&=
\int_{\Omega\setminus\ol{D}} \nabla w \cdot\nabla u_0 -  \int_{\doo D}\frac{\doo w}{\doo\nu} \ u_0 \ \der S.
\end{split}
\end{equation}
Subtracting~\eqref{ineA} from~\eqref{ineB} multiplied by two, we obtain
\begin{equation}
\begin{split}
\langle (\Lambda_D - \Lambda_{\emptyset}) f, f \rangle
&\leq
\int_{D}\abs{\nabla u_0}^2 \der x - 2 \int_{\doo D}\frac{\doo w}{\doo\nu} \ u_0 \ \der S.
\end{split}
\end{equation}

\textbf{Proof of~\eqref{eq:estimate-ii}:}
We observe that
\begin{equation}
\label{a}
\begin{split}
&\langle (\Lambda_D - \Lambda_{\emptyset}) f, f \rangle \\
& = \int_{\Omega\setminus\ol{D}} \abs{\nabla u}^2 \der x - \int_{\Omega} \abs{\nabla u_0}^2 \der x \\
& = \int_D \abs{\nabla u_0}^2 \der x - I ,
\end{split}
\end{equation}
where
\begin{equation}
\label{low1}
\begin{split}
I
&=
2\int_D \abs{\nabla u_0}^2 \der x + \int_{\Omega\setminus\ol{D}} \abs{\nabla u_0}^2 \der x - \int_{\Omega\setminus\ol{D}} \abs{\nabla u}^2 \der x
\\
& = 2\int_{\Omega} \abs{\nabla u_0}^2 \der x - \int_{\Omega\setminus\bar{D}} \abs{\nabla u}^2 \der x + \int_{\Omega\setminus\ol{D}} \abs{\nabla u_0}^2 \der x
\\&\quad- 2\int_{\Omega\setminus\ol{D}} \abs{\nabla u_0}^2 \der x \\
& = 2\int_{\Omega}\nabla u_0 \cdot \nabla u \der x - \int_{\Omega\setminus\bar{D}} \abs{\nabla u}^2\der x - \int_{\Omega\setminus\ol{D}} \abs{\nabla u_0}^2 \der x \\
& = 2\int_{\Omega\setminus\ol{D}}\nabla u_0 \cdot \nabla u \der x - \int_{\Omega\setminus\ol{D}} \abs{\nabla u}^2\der x - \int_{\Omega\setminus\ol{D}} \abs{\nabla u_0}^2 \der x
\\
&=
-\int_{\Omega\setminus\bar D}
\abs{\nabla u-\nabla u_0}^2\der x
.
\end{split}
\end{equation}
The third identity follows from equation~\eqref{free} and the second last equality is due to the fact that $\nabla u = 0$ in~$D$.
Since $I\leq0$, the required inequality now follows from~\eqref{a}.
\end{proof}

\begin{lemma}
\label{lma:H1-estimates}
There is a constant~$C$ independent of~$f$ so that
\begin{equation}
\label{eq:u-H1}
\aabs{u}_{H^1(\Omega)}
\leq
C\aabs{f}_{H^{1/2}(\partial\Omega)}
\end{equation}
and
\begin{equation}
\label{eq:w-H1}
\aabs{w}_{H^1(\Omega)}
\leq
C\aabs{f}_{H^{1/2}(\partial\Omega)}.
\end{equation}
\end{lemma}

\begin{proof}
In this proof~$C$ will always denote a universal constant independent of~$f$, but the constant may be different in different estimates.

Since $w=u-u_0$, the triangle inequality gives
\begin{equation}
\aabs{u}_{H^1(\Omega)}
\leq
\aabs{u_0}_{H^1(\Omega)}+\aabs{w}_{H^1(\Omega)}.
\end{equation}
The estimate $\aabs{u_0}_{H^1(\Omega)}\leq C\aabs{f}_{H^{1/2}(\partial\Omega)}$ is a classical continuity estimate for the solution operator of the Laplacian, so~\eqref{eq:w-H1} implies~\eqref{eq:u-H1}.

Since~$w$ has zero boundary values in the Sobolev sense, it suffices to establish
\begin{equation}
\label{eq:w-L2}
\aabs{\nabla w}_{L^2(\Omega)}
\leq
C\aabs{f}_{H^{1/2}(\partial\Omega)}.
\end{equation}
Using the triangle inequality and the classical estimate for~$\aabs{u_0}_{H^1(\Omega)}$ again, we find that it suffices to prove
\begin{equation}
\label{eq:u-L2}
\aabs{\nabla u}_{L^2(\Omega)}
\leq
C\aabs{f}_{H^{1/2}(\partial\Omega)}.
\end{equation}
We can establish~\eqref{eq:u-L2} using the fact that~$u$ solves~\eqref{linear_1}.
We have $\Delta u=0$ in~$\Omega\setminus\bar D$ and $\nabla u=0$ in~$D$, so
\begin{equation}
\begin{split}
\aabs{\nabla u}_{L^2(\Omega)}^2
&=
\int_{\Omega\setminus\bar D}\abs{\nabla u}^2
=
\int_{\partial\Omega}f\partial_\nu u.
\end{split}
\end{equation}
The boundary term on~$\partial D$ and the interior term in~$\Omega\setminus\bar D$ vanish due to~\eqref{linear_1}.
Now estimating
\begin{equation}
\int_{\partial\Omega}f\partial_\nu u
\leq
C\aabs{f}_{H^{1/2}(\partial\Omega)}\aabs{\nabla u}_{L^2(\Omega)}
\end{equation}
finishes the proof of~\eqref{eq:u-L2} and therefore also those of~\eqref{eq:w-L2}, \eqref{eq:w-H1} and~\eqref{eq:u-H1}.
\end{proof}

The following lemma is a consequence of the previous one.
We do not need it, but we record it here for completeness.

\begin{lemma}
\label{lma:continuous-constant}
Let the connected components of~$D$ be $D_1,\dots,D_K$.
For each $k\in\{1,\dots,K\}$ the function that sends the boundary data~$f$ to the constant value~$c^f_k$ the solution attains on~$D_k$ is linear and continuous $H^{1/2}(\partial\Omega)\to\R$.
\end{lemma}

\begin{proof}
Linearity is an elementary observation, and continuity follows from lemma~\ref{lma:H1-estimates}.
Namely, estimate~\eqref{eq:u-H1} yields
\begin{equation}
\aabs{u}_{L^2(D_k)}
\leq
\aabs{u}_{H^1(\Omega)}
\leq
C
\aabs{f}_{H^{1/2}(\partial\Omega)}.
\end{equation}
Since~$u$ is constant in each~$D_k$, this proves the lemma.
\end{proof}

\subsection{Layer potentials and a boundary integral estimate}\label{boundryEstimaTE}
In this subsection our main aim is to prove the following proposition.
Let $\Omega\subset\mathbb{R}^n$, $n\geq2$, be a bounded $C^1$-smooth domain and $D\subset\mathbb{R}^n$ be a bounded and connected subset of~$\Omega$ with Lipschitz regular boundary. 

\begin{proposition}\label{proENclo}
Let the reflected solution $w=u-u_0$ satisfies~\eqref{refl_lin}.
We have the following integral estimate:
\begin{equation}
\abs{\int_{\doo D}\frac{\doo w}{\doo\nu} u_0 \der S} \leq C \|u_0\|_{H^1(D)}^{2},
\end{equation}
where $C>0$ be a constant independent of~$u_0$.
\end{proposition}

The main ingredient to show the above estimate is first to write the solution of the problem~\eqref{refl_lin} in terms of the single layer potential and then to use the properties of the layer potential operator in the appropriate Sobolev spaces.
Before going into details of the proof of this proposition, we would like to recall the integral operators of single layer type and their mapping properties.


Let $\mathcal{E}\subset \mathbb{R}^n, n\geq2$, be an open bounded Lipschitz domain. We denote interior and exterior region by $\mathcal{E}_+ \coloneqq \mathcal{E}$ and $\mathcal{E}_- \coloneqq \mathbb{R}^n\setminus\overline{\mathcal{E}}$ respectively.
We first define the single layer potential operator corresponding to the boundary~$\doo\mathcal{E}$ as follows:
\begin{equation}
\mathcal{S}_{\doo\mathcal{E}}\hat{q}(x) \coloneqq \int_{\doo\mathcal{E}} \Phi(x-y)\hat{q}(y) \der S(y),
\end{equation}
for $x\in \mathbb{R}^n\setminus\doo\mathcal{E}$, where $\hat{q}\in H^{-\frac{1}{2}}(\doo\mathcal{E})$ and~$\Phi$ is the fundamental solution for the Laplacian. If $n = 2$, we choose $r > \diam \Omega$ and $\Phi = \frac{1}{2\pi}\log \frac{r}{\abs{x}}$.
For any $\hat{q}\in H^{-\frac{1}{2}}(\doo\mathcal{E})$, the function $\mathcal{S}_{\doo\mathcal{E}}\hat{q}$ solves the Laplace equation $\Delta(\mathcal{S}_{\doo\mathcal{E}}\hat{q}) = 0$ in the region $\mathbb{R}^n\setminus{\doo\mathcal{E}}$. It is well known that the single layer potential operator defines a bounded linear operator from $H^{-\frac{1}{2}}(\doo\mathcal{E})$ to $H_{loc}^{1}(\mathbb{R}^n)$, see for instance~\cite[Theorem~6.11]{McLean:2000}.
We now define the function space, for $n\geq 2$,
\begin{equation}
\hat{H}(\doo\mathcal{E}) \coloneqq \{\hat{q} \in  H^{-\frac{1}{2}}(\doo\mathcal{E}) ; \ip{\hat{q}}{1}_{-\frac{1}{2},\frac{1}{2},\doo\mathcal{E}} = 0\}, 
\end{equation}
where $\ip{\hat{q}}{1}_{-\frac{1}{2},\frac{1}{2},\doo\mathcal{E}}$ represents the duality product between $H^{-\frac{1}{2}}(\doo\mathcal{E})$ and $H^{\frac{1}{2}}(\doo\mathcal{E})$. The (interior/exterior) non-tangential boundary traces of~$\mathcal{S}_{\doo\mathcal{E}}\hat{q}$ are given by 
\begin{equation}
\lim_{\substack{x\rightarrow x_0 \\ x \in \Gamma_{+}(x_0)}} \mathcal{S}_{\doo\mathcal{E}}\hat{q}(x)
= \lim_{\substack{x\rightarrow x_0 \\ x \in \Gamma_{-}(x_0)}} \mathcal{S}_{\doo\mathcal{E}}\hat{q}(x)
= S_{\doo\mathcal{E}}\hat{q}(x_0), \  \ x_0\in\doo\mathcal{E},
\end{equation}
where $S_{\doo\mathcal{E}}\hat{q}$ is the trace of the single layer operator $\mathcal{S}_{\doo\mathcal{E}}\hat{q}$ on~$\doo\mathcal{E}$, i.e.,
\begin{equation}
S_{\doo\mathcal{E}}\hat{q}(x) \coloneqq \int_{\doo\mathcal{E}} \Phi(x-y)\hat{q}(y) \der S(y), \  \ x\in\doo\mathcal{E},  
\end{equation}
and~$\Gamma_{\pm}$ denote the interior of the two components (in~${\mathcal{E}}_{+}$ and in~${\mathcal{E}}_{-}$) of a regular family of circular doubly truncated cones $\{\Gamma(x) ; x\in \doo\mathcal{E} \} $ with vertex at~$x$ (for definition see for example~\cite[section 0.4]{Verchota:1984}).
In addition, for almost any $x_0\in \doo\mathcal{E}$,
\begin{equation}
\begin{split}
\lim_{\substack{x\rightarrow x_0 \\ x \in \Gamma_{\pm}(x_0)}} \frac{\doo\mathcal{S}_{\doo\mathcal{E}}\hat{q}(x)}{\doo\nu}
&\coloneqq \lim_{\substack{x\rightarrow x_0 \\ x \in \Gamma_{\pm}(x_0)}} \ip{\nu(x_0)}{\nabla\mathcal{S}_{\doo\mathcal{E}}\hat{q}(x)}
\\&\eqqcolon (\pm\frac{1}{2} I + \mathcal{K}^{*})\hat{q}(x_0)
\end{split}
\end{equation}
where~$\mathcal{K}^{*}$ is the adjoint of the double layer potential operator, see for instance~\cite{Verchota:1984}.

We now define the equilibrium density for~$\doo\mathcal{E}$ as in~\cite[Theorem~8.15]{McLean:2000}.
There exists a unique distribution $\psi_{eq} \in H^{-\frac{1}{2}}(\doo\mathcal{E})$ so that~$S_{\doo\mathcal{E}}\psi_{eq}$ is constant on~$\doo\mathcal{E}$ and $\ip{\psi_{eq}}{1}_{-\frac{1}{2}, \frac{1}{2}, \doo\mathcal{E}} = 1$.
In addition, if $n\geq 3$, then $S_{\doo\mathcal{E}}\psi_{eq} >0$.
We now define the capacity of the boundary in terms of the equilibrium density. If $n\geq 3,$ the capacity of $\doo\mathcal{E}$ is denoted as $\text{Cap}(\doo\mathcal{E})$ and defined as $\text{Cap}(\doo\mathcal{E}) = \frac{1}{S_{\doo\mathcal{E}}\psi_{eq} }$ and
$\text{Cap}(\doo\mathcal{E}) = \exp(-2\pi S_{\doo\mathcal{E}}\psi_{eq})$ when $n = 2$.
Following McLean~\cite{McLean:2000}, see also~\cite{Munnier:Ramdani:2017}, we define
\begin{equation}
H(\doo\mathcal{E}) \coloneqq \{q\in H^{\frac{1}{2}}(\doo\mathcal{E}) ; \ip{\psi_{eq}}{q}_{-\frac{1}{2},\frac{1}{2},\doo\mathcal{E}} =0 \}.
\end{equation}
Note that the operator 
\begin{equation}\label{single1}
S_{\doo\mathcal{E}} \colon  \hat{H}(\doo\mathcal{E}) \rightarrow H(\doo\mathcal{E})
\end{equation}
is defined by 
\begin{equation}\label{single2}
q \coloneqq S_{\doo\mathcal{E}} \hat{q}. 
\end{equation}
For $n=2$, the operator $S_{\doo\mathcal{E}}$ is an isomorphism if $\text{Cap}(\doo\mathcal{E}) \neq r$, see~\cite[Theorem~8.16]{McLean:2000}.
When $n \geq 3$, $S_{\doo\mathcal{E}}$ is injective~\cite[corollary 8.11]{McLean:2000}.
Also, it is of Fredholm index zero~\cite[theorem 7.6]{McLean:2000}. Hence $S_{\doo\mathcal{E}}$ is an isomorphism.
Note that we do not need any assumption on the capacity of the domain for $n\geq 3$.
With respect to the inner products 
\begin{equation}
\ip{\hat{q}}{\hat{p}}_{-\frac{1}{2},\doo\mathcal{E}} = \ip{q}{p}_{\frac{1}{2},\doo\mathcal{E}} = \ip{\hat{q}}{p}_{-\frac{1}{2},\frac{1}{2},\doo\mathcal{E}}, \  \ \forall \ \hat{q}, \hat{p} \in \hat{H}(\doo\mathcal{E}),
\end{equation}
the above isomorphism is actually an isometry through the identities
\begin{equation}\label{IDisometry}
\|\hat{q}\|_{-\frac{1}{2},\doo\mathcal{E}}^{2} = \|q\|_{\frac{1}{2},\doo\mathcal{E}}^{2}
 = \int_{\mathbb{R}^n}\abs{\nabla(\mathcal{S}_{\doo\mathcal{E}}\hat{q})}^2 \der x, \ \forall \ \hat{q} \in \hat{H}(\doo\mathcal{E}).
\end{equation}
The second equality can be realised by using the asymptotic behaviour of the single layer potential at infinity 
\begin{equation}\label{asymp}
\mathcal{S}_{\mathcal{E}}\hat{q}(x) =
\begin{cases}
\frac{1}{2\pi}  \ip{\hat{q}}{1}_{-\frac{1}{2},\frac{1}{2},\doo\mathcal{E}} \log{\frac{r}{\abs{x}}} + \mathcal{O}(\abs{x}^{-1}) & \text{ when } n=2\\ \\
\mathcal{O}(\abs{x}^{2-n}) & \text{ when } n \geq 3
\end{cases}
\end{equation}
and Green's formula, see~\cite[Theorem 8.12]{McLean:2000}.
We also need the following orthogonal decomposition in our analysis. Let us first define an orthogonal projection 
\begin{equation}
\Pi_{\doo\mathcal{E}} \colon H^{\frac{1}{2}}(\doo\mathcal{E}) \rightarrow H(\doo\mathcal{E}),
\end{equation}
which is a bounded linear operator with the following unique decomposition
\begin{equation}
q = \ip{\psi_{eq}}{q}1 + q_0,
\end{equation}
for all $q\in H^{\frac{1}{2}}(\doo\mathcal{E}),$ where $q_0 \coloneqq \Pi_{\doo\mathcal{E}}q \in H(\doo\mathcal{E}),$ see~\cite{Munnier:Ramdani:2017}. Denote by~$\Tr_{\doo\mathcal{E}}$, the trace operator mapped into~$H^{\frac{1}{2}}(\doo\mathcal{E})$ and also $\Tr_{\doo\mathcal{E}}^{0} \coloneqq \Pi_{\doo\mathcal{E}}(\Tr_{\doo\mathcal{E}})$.

Let us assume that the diameter of~$\Omega$ is less than~$r$ only when~$n=2$.
In the higher dimensional case we do not need the assumption since the operator $S_{\doo\mathcal{E}}$ is an isomorphism even without any assumptions on the capacity of the domain.
In our setting~$D$ is a bounded Lipschitz domain included in~$\Omega$. So, the assumption on the diameter of~$\Omega$ in $\mathbb{R}^2$ implies that $\text{Cap}(\doo\Omega)<r$ and $\text{Cap}(\doo D)<r$~\cite[exercise 8.12]{McLean:2000}.
Recall the model~\eqref{refl_lin} in the case where~$D$ is connected.
The reflected solution $w = u- u_0$ satisfies
\begin{equation}\label{refl_lin2}
\begin{cases}
\Delta w = 0 \ \text{in}\ \Omega\setminus\bar{D} \\
w=c_f-u_0 \ \text{in}\ \bar{D} \\
\int_{\doo D}\frac{\doo w}{\doo\nu} \der S = -\int_{\doo D}\frac{\doo u_0}{\doo\nu} \der S \\
w= 0 \ \text{on}\ \partial \Omega.
\end{cases}
\end{equation} 
Here~$c_f$ is a constant depending on~$f$.
(This dependence is linear and continuous by lemma~\ref{lma:continuous-constant}.)
To solve the problem~\eqref{refl_lin2}, we use the integral equation method following~\cite[Proposition~2.12]{Munnier:Ramdani:2017}. We represent the solution~$w$ as the superposition of single layer potentials on the boundaries~$\doo D$ and~$\doo\Omega$ as follows:
\begin{equation}
w(x) \coloneqq \mathcal{S}_{\doo D}\hat{p} + \mathcal{S}_{\doo \Omega}\hat{q}, \  \  x\in D\cup(\Omega\setminus\overline{D}),
\end{equation}
where $\hat{p}\in \hat{H}(\doo D)$ and $\hat{q}\in H^{-\frac{1}{2}}(\doo\Omega).$ The density functions~$\hat{p}$ and~$\hat{q}$ satisfy the following system of integral equations on~$\doo D$ and~$\doo\Omega$:
\begin{equation}\label{inteSysD} 
p + \Tr_{\doo D}(\mathcal{S}_{\doo\Omega}\hat{q}) = c_f - u_0 \  \ \text{on}\ \doo D 
\end{equation}
\begin{equation}\label{inteSysO} 
\Tr_{\doo\Omega}(\mathcal{S}_{\doo D}\hat{p}) + q = 0 \  \ \text{on}\ \doo\Omega. 
\end{equation}
We define here the boundary interaction operators~$K_{\doo\Omega}^{\doo D}$ and~$K_{\doo D}^{\doo\Omega}$ between~$\doo\Omega$ and~$\doo D$ as follows:
\begin{equation}
K_{\doo\Omega}^{\doo D} \colon  H(\doo\Omega) \rightarrow H(\doo D)
\end{equation}
by
\begin{equation}
K_{\doo\Omega}^{\doo D}(q) \coloneqq \Tr_{\doo D}^{0}(\mathcal{S}_{\doo\Omega}\hat{q})
\end{equation}
and
\begin{equation}
K_{\doo D}^{\doo\Omega} \colon  H(\doo D) \rightarrow  H(\doo\Omega)
\end{equation}
by
\begin{equation}
K_{\doo D}^{\doo\Omega}(p) \coloneqq \Tr_{\doo\Omega}^{0}(\mathcal{S}_{\doo D}\hat{p}),
\end{equation}
for more details see~\cite{Munnier:Ramdani:2017}.

Before presenting the proof of Proposition~\ref{proENclo}, we state the following lemma.

\begin{lemma}
\label{property_layer}
The operators defined above have the following properties:
\begin{enumerate}[(i)]
	\item If $p\in H(\doo D)$, then $q:= \Tr_{\doo\Omega}(\mathcal{S}_{\doo D}\hat{p}) \in H(\doo\Omega)$.
	\item The norms of the operators $K_{\doo\Omega}^{\doo D}$ and $K_{\doo D}^{\doo\Omega}$ are strictly less than~$1$.
\end{enumerate}
\end{lemma}

\begin{proof}
(i) The first part of the lemma is proved in~\cite[Proposition~2.9]{Munnier:Ramdani:2017} for a planar domain.
However, for $n\geq 3$, we have the asymptotic behaviour of the single layer potential at infinity of the form \eqref{asymp}.
In this case, the proof goes similarly to~\cite[Proposition~2.9]{Munnier:Ramdani:2017}.

(ii)  For $n=2$, the proof follows from~\cite[Proposition~2.11]{Munnier:Ramdani:2017}.
In higher dimensions the proof goes along the same lines as in~\cite[Proposition~2.11]{Munnier:Ramdani:2017}.
For the convenience of the reader we present more details.
To prove the appropriate norm bounds for the operators $K_{\doo\Omega}^{\doo D}$ and $K_{\doo D}^{\doo\Omega}$, we first define certain kind of quotient weighted Sobolev spaces:
\begin{equation}
W_{0}^{1}(\mathbb{R}^n) := \{u\in \mathcal{D}'(\mathbb{R}^n) ; \rho u\in L^2(\mathbb{R}^n), \nabla u \in L^2(\mathbb{R}^n) \}/\mathbb{R}
\end{equation}
where the weight is given by
\begin{equation}
\rho(x) =
\begin{cases}
(\sqrt{1+\abs{x}^2} \log(2+\abs{x}^2))^{-1} & \text{when } n=2\\ \\
({1+\abs{x}^2})^{-1/2} & \text{when } n \geq 3.
\end{cases}
\end{equation}
Under the inner product
\begin{equation}
\langle u,v \rangle_{W_{0}^{1}(\mathbb{R}^n)} := \int_{\mathbb{R}^n} \nabla u \cdot \nabla v \der x,
\end{equation}
the space $W_{0}^{1}(\mathbb{R}^n)$ is a Hilbert space.
For any $q\in H(\doo\Omega)$, we have
\begin{equation}
\label{opNorM}
\begin{split}
&\norm{K_{\doo\Omega}^{\doo D}  q}_{1/2,\doo D}
= \\
& \quad \inf \left\{\|u\|_{W_{0}^{1}(\mathbb{R}^n)} ; u \in W_{0}^{1}(\mathbb{R}^n),
\Tr_{\doo D}^{0}u = \Tr_{\doo D}^{0} (\mathcal{S}_{\doo\Omega} \hat{q})\right\}. 
\end{split}
\end{equation}
The above identity can be proved by considering an orthogonal decomposition of the Hilbert space 
$W_{0}^{1}(\mathbb{R}^n)$ in terms of the single layer potential operator.
See~\cite[Proposition~2.7]{Munnier:Ramdani:2017} for the proof in $\mathbb{R}^2$; the same technique works in higher dimensions.
Therefore
\begin{equation}
\|K_{\doo\Omega}^{\doo D} q\|_{1/2,\doo D} \leq \|\mathcal{S}_{\doo\Omega}\hat{q}\|_{W_{0}^{1}(\mathbb{R}^n)} = \|q\|_{1/2,\doo\Omega},
\end{equation}
which implies that $\|K_{\doo\Omega}^{\doo D}\| \leq 1$.
Since $K_{\doo\Omega}^{\doo D}$ is compact, the norm~\eqref{opNorM} is achieved for some $q_{\doo\Omega} \in H(\doo\Omega)$.
If we suppose
\begin{equation}
\|K_{\doo\Omega}^{\doo D} q_{\doo\Omega}\|_{1/2,\doo D} = \|q_{\doo\Omega}\|_{1/2,\doo\Omega}
\end{equation}
then $\mathcal{S}_{\doo D}\hat{q}_{\doo D} = \mathcal{S}_{\doo\Omega}\hat{q}_{\doo\Omega}$ in $\R^n, n\geq 3$, where $q_{\doo D} := K_{\doo\Omega}^{\doo D} q_{\doo\Omega}$. Finally using the jump relation of the derivative of the single layer potential, we have $\hat{q}_{\doo\Omega}=0$, which gives a contradiction.
Thus, $\norm{K_{\doo\Omega}^{\doo D}} < 1 $ and the proof is complete.
\end{proof}

According to lemma~\ref{property_layer},
we have $\Tr_{\doo\Omega}(\mathcal{S}_{\doo D}) \in H(\doo\Omega)$, which implies \mbox{$\Tr_{\doo\Omega}(\mathcal{S}_{\doo D}) = K_{\doo D}^{\doo\Omega}p$}.
Therefore, we deduce from~\eqref{inteSysO} that $K_{\doo D}^{\doo\Omega}p + q = 0$. Now applying the projection~$\Pi_{\doo D}$ to~\eqref{inteSysD}, we obtain the following system of integral equations:
\begin{align}\label{inteSysD1} 
p + K_{\doo\Omega}^{\doo D}q &= - \Pi_{\doo D}(u_0) &\text{ on }\doo D  \\
\label{inteSysO2} 
K_{\doo D}^{\doo\Omega}p + q &= 0 &\text{ on } \doo\Omega. 
\end{align}
Replacing~$q$ by $-K_{\doo D}^{\doo\Omega}p$ in~\eqref{inteSysD1}, we obtain 
\begin{equation}
(I-K)p =  - \Pi_{\doo D}(u_0) \  \ \text{on}\ \doo D, 
\end{equation}
where $K \coloneqq K_{\doo\Omega}^{\doo D}K_{\doo D}^{\doo\Omega}$ defines a bounded linear map from~$H(\doo D)$ to~$H(\doo D)$.
From lemma~\ref{property_layer} one can observe that the operator~$K$ is a contraction map and hence $I-K$ is invertible. Therefore, using identity~\eqref{IDisometry}, we have
\begin{equation}\label{proj_Inq}
\|\hat{p}\|_{-\frac{1}{2},\doo D} = \|p\|_{\frac{1}{2},\doo D} \leq C \|\Pi_{\doo D}(u_0)\|_{\frac{1}{2},\doo D},
\end{equation}
where the constant $C>0$ does not depend on~$u_0$.

We are now in a position to present the details of the proof of proposition~\ref{proENclo}.

\begin{proof}[Proof of proposition~\ref{proENclo}]
For our convenience we use the following notation: $w_{+}(x) \coloneqq w(x)$ if $x\in D$ and $w_{-}(x) \coloneqq w(x)$ if $x\in \Omega\setminus\overline{D}$. Fix a point~$x_0$ on~$\doo D$. We now calculate the jump of the Neumann trace at the point~$x_0$ from the interior and exterior of the boundary~$\doo D$. The Neumann trace of~$w_{+}$ from inside~$D$ at the point~$x_0$ is
\begin{equation}
\begin{split}
\frac{\doo w_{+}(x_0)}{\doo\nu} 
& = \frac{\doo}{\doo\nu}(\mathcal{S}_{\doo D}\hat{p})|_{x_0}^{+} + \frac{\doo}{\doo\nu}(\mathcal{S}_{\doo\Omega}\hat{q}(x_0)) \\
& = \left(\frac{1}{2}I + \mathcal{K}^{*}\right) \hat{p}(x_0) + \frac{\doo}{\doo\nu}(\mathcal{S}_{\doo\Omega}\hat{q}(x_0)),
\end{split}
\end{equation}
where $\frac{\doo}{\doo\nu}(\mathcal{S}_{\doo D}\hat{p})|_{x_0}^{+} $ denotes the Neumann trace of the single layer potential at~$x_0$ from the inside~$D$. More precisely, for $ x_0\in \doo D$,
\begin{equation}
\frac{\doo}{\doo\nu}(\mathcal{S}_{\doo D}\hat{p})|_{x_0}^{+} \coloneqq \lim_{\substack{x\rightarrow x_0 \\ x \in \Gamma_{+}(x_0)}} \frac{\doo\mathcal{S}_{\doo D}\hat{p}(x)}{\doo\nu}.
\end{equation}
Similarly, the Neumann trace of~$w_{-}$ from the exterior of~$D$ at~$x_0$ is
\begin{equation}
\begin{split}
\frac{\doo w_{-}(x_0)}{\doo\nu} 
& = \frac{\doo}{\doo\nu}(\mathcal{S}_{\doo D}\hat{p})|_{x_0}^{-} + \frac{\doo}{\doo\nu}(\mathcal{S}_{\doo\Omega}\hat{q}(x_0)) \\
& = \left(- \frac{1}{2}I + \mathcal{K}^{*}\right) \hat{p}(x_0) + \frac{\doo}{\doo\nu}(\mathcal{S}_{\doo\Omega}\hat{q}(x_0)),
\end{split}
\end{equation}
where $\frac{\doo}{\doo\nu}(\mathcal{S}_{\doo D}\hat{p})|_{x_0}^{-} $ denotes the Neumann trace of the single layer potential at~$x_0$ from the inside~$\Omega\setminus\overline{D}$. More precisely, for $x_0\in \doo D$,
\begin{equation}
\frac{\doo}{\doo\nu}(\mathcal{S}_{\doo D}\hat{p})|_{x_0}^{-} \coloneqq \lim_{\substack{x\rightarrow x_0 \\ x \in \Gamma_{-}(x_0)}} \frac{\doo\mathcal{S}_{\doo D}\hat{p}(x)}{\doo\nu}.
\end{equation}
Therefore, the jump of the Neumann trace at the point~$x_0$ is given by
\begin{equation}
\frac{\doo w_{+}(x_0)}{\doo\nu} - \frac{\doo w_{-}(x_0)}{\doo\nu} = \hat{p}(x_0).
\end{equation}
Finally, the boundary integral $\int_{\doo D}\frac{\doo w}{\doo\nu} u_0 \der S$ can be estimated as follows:
\begin{equation}
\begin{split}
\abs{\int_{\doo D}\frac{\doo w}{\doo\nu} u_0 \der S}
& = \abs{\int_{\doo D}\frac{\doo w_{-}}{\doo\nu} u_0 \der S} \\
& \leq
\bigg\lvert{\underbrace{\int_{\doo D}{\frac{\doo w_{+}}{\doo\nu} }{u_0} \der S}_{\eqqcolon I_1}}\bigg\rvert
+
\bigg\lvert{\underbrace{\int_{\doo D}{\hat{p}}{u_0} \der S}_{\eqqcolon I_2}}\bigg\rvert
.
\end{split}
\end{equation}
Note that
$$\frac{\doo w_{+}}{\doo\nu}|_{\doo D}  = -\frac{\doo u_0}{\doo\nu}|_{\doo D},$$
 so applying the trace theorem we estimate integral~$I_1$ by
\begin{equation}
\abs{I_1}
\leq \left\|\frac{\doo u_0}{\doo\nu}\right\|_{H^{-\frac{1}{2}}(\doo D)} \|u_0\|_{H^1(D)}
\leq C \|u_0\|_{H^1(D)}^{2}.
\end{equation}
Using~\eqref{proj_Inq} and the trace theorem, we obtain
\begin{equation}
\begin{split}
\abs{I_2} 
& \leq \|\hat{p}\|_{H^{-\frac{1}{2}}(\doo D)} \|u_0\|_{H^{\frac{1}{2}}(\doo D)} \\
& \leq \|\hat{p}\|_{-\frac{1}{2},\doo D} \|u_0\|_{H^{\frac{1}{2}}(\doo D)} \\
& \leq C \|\Pi_{\doo D}(u_0)\|_{\frac{1}{2},\doo D} \|u_0\|_{H^1(D)} \\
& \leq C \left[\|u_0\|_{H^{\frac{1}{2}}(\doo D)} + C \|\psi_{eq}\|_{H^{-\frac{1}{2}}(\doo D)} \|u_0\|_{H^{\frac{1}{2}}(\doo D)}\right]\|u_0\|_{H^1(D)} \\
& \leq C \|u_0\|_{H^1(D)}^{2}.
\end{split}
\end{equation}
Hence
\begin{equation}
\abs{\int_{\doo D}\frac{\doo w}{\doo\nu}u_0} \leq C \|u_0\|_{H^1(D)}^{2},
\end{equation}
where~$C$ is a positive constant independent of~$u_0$.
\end{proof}

\subsection{The enclosure method}
\label{sec:p2-enclosure}

We start with the enclosure method.
The basic idea of the this method is to use the complex geometrical optic (CGO) solution as a test function instead of using the fundamental solution of the Laplace operator.
We follow the approach of Ikehata~\cite{Ikehata:1999:jan}.
We mainly use the CGO solution with linear phase function for detecting the convex hull of the obstacle.
Here it is more convenient to work with complex-valued functions, but in the linear case $p=2$ the real and complex versions of the inverse boundary value problem are trivially equivalent.

CGO solutions with linear phase for the Laplacian can be constructed as follows.
Take any unit vector $\rho\in\R^n$ and an orthogonal unit vector $\rho^\perp\in\R^n$.
Let $t$ and $\tau$ be real numbers with $\tau>0$.
Then the function $u_0(x)=\exp(\tau(x\cdot\rho+ix\cdot\rho^\perp-t))$ is harmonic.
Given $\rho$, we consider $\rho^\perp$ fixed although there is a freedom of choice; the results do not depend on the choice and we therefore omit $\rho^\perp$ in our notation.

We define an indicator function as
\begin{equation}\label{Indicator:p=2}
I_{\rho}(u_0,\tau,t) \coloneqq \tau^{n-2} \int_{\doo\Omega}\ol{(\Lambda_D - \Lambda_{\emptyset})(u_0)} u_0 \der S,
\end{equation}
where $\rho\in\sphere^{n-1}, \tau>0, t\in\R$ and $\Lambda_D$, $\Lambda_{\emptyset}$ are the DN maps for the problems~\eqref{linear_1} and~\eqref{linear_2} respectively.
For $\rho\in\sphere^{n-1}$ we define the convex support function~$h_D(\rho)$ of~$D$ as
\begin{equation}
\label{eq:h-def}
h_D(\rho) \coloneqq \sup_{x\in D} x\cdot\rho.
\end{equation}
One can reconstruct the convex hull of the superconductive material~$D$ from the following behavior of the indicator function.

\begin{enumerate} 
\item When $t>h_D(\rho)$ we have
\begin{equation}
\lim_{\tau \to \infty}\vert I_{\rho}(u_0,\tau,t)\vert = 0,
\end{equation}
and more precisely,
\begin{equation}
\label{main_behavior_1far}
\vert I_{\rho}(u_0,\tau,t)\vert \leq Ce^{-c\tau}
\end{equation}
for $\tau > 0$, and with~$c,C>0$.
\item When $t = h_D(\rho)$ we have
\begin{equation}
\liminf_{\tau \to \infty}\abs{I_{\rho}(u_0,\tau, h_{D}(\rho))} >0,
\end{equation}
and more precisely
\begin{equation}
\label{main_behavior_2far}
c \leq  \vert I_{\rho}(u_0,\tau,h_D(\rho)) \vert \leq C \tau^n
\end{equation}
where $c, C>0$, and $\tau >0$ for the upper bound, and $\tau \gg 1$ for the lower bound.
\item When $t < h_D(\rho)$ we have
\begin{equation}
\lim_{\tau \to \infty}\abs{I_{\rho}(u_0,\tau,t)} = \infty,
\end{equation}
and more precisely,
\begin{equation}\label{main_behavior_3far}
\vert I_{\rho}(u_0,\tau,t) \vert \geq Ce^{c\tau},
\end{equation}
where $\tau \gg1$ and $c,C>0$.
\end{enumerate}

Note that $I_{\rho}(u_0, \tau, t) = \exp(2\tau(h_D(\rho)-t)) I_{\rho}(u_0, \tau, h_D(\rho))$.
Therefore to prove the above properties of the indicator function it is enough to prove the inequality~\eqref{main_behavior_2far}.
To do that, we state the following lemma, see~\cite[Lemma~4.7]{Brander:Kar:Salo:2015} for the proof.

\begin{lemma}\label{estimateCGO:p=2}
Let $D\subset\Omega$ be a Lipschitz domain. Then, for $\tau \gg 1$ we have
\begin{equation}
\int_D \exp(-2\tau(h_D(\rho)-x\cdot\rho)) \der x \geq C\tau^{-n}.
\end{equation}
\end{lemma} 

Therefore estimate~\eqref{main_behavior_2far} follows from proposition~\ref{proENclo} and the properties of the CGO solutions as described in lemma~\ref{estimateCGO:p=2}.
Using the above type of CGO solution, we describe a reconstruction procedure for the enclosure method as follows:\\
\textbf{Step 1.} For each direction $\rho\in\sphere^{n-1}$, we first define an indicator function $I_{\rho}(u_0, \tau, t)$ via the use of CGO solution and the corresponding measurement data (DN map) as in~\eqref{Indicator:p=2}. \\
\textbf{Step 2.} When the hyperplane (level set) $\{x\in\R^n ; x\cdot\rho = t\}$ moves along~$\rho$ for each~$\rho$ and~$t$, observe the asymptotic behavior of~$I_{\rho}(u_0, \tau, t)$ as $\tau\gg 1$ as in~\eqref{main_behavior_1far},~\eqref{main_behavior_2far} and~\eqref{main_behavior_3far}. Then one can know whether or not the level set of the phase function of the CGO solution touches the interface of the obstacle. \\
\textbf{Step 3.} Now, the support function~$h_D(\rho)$
can be determined by the following formula:
\begin{equation}
h_D(\rho) - t = \lim_{\tau\to\infty} \frac{\log\abs{I_{\rho}(u_0, \tau, t)}}{2\tau}.
\end{equation}
\textbf{Step 4.} Finally, the convex hull of~$D$ can be reconstructed by taking the intersection of $\{x\in\sphere^n ; x\cdot\rho \leq h_D(\rho)\}$ in a dense set of directions $\rho\in\sphere^{n-1}$.

Using the enclosure method we have thus proven the following proposition.
Note that it is weaker than theorem~\ref{thm:p2-inverse}.
Both results are constructive, as their proofs come with a method for finding the inclusion or its convex hull.

\begin{proposition}
\label{prop:p2-enclosure}
Let $\Omega\subset\R^n$ be a bounded Lipschitz domain and $D\subset\Omega$ a connected and compactly contained smaller Lipschitz domain.
Consider a conductivity~$\sigma$ of the form
\begin{equation}
\sigma(x) =
\begin{cases}
1 & \text{ for } x \in \Omega\setminus \bar{D} \\
\infty & \text{ for } x \in D
\end{cases}
\end{equation}
and the corresponding DN map $\Lambda_D=\Lambda_\sigma$.
The function~$h_D$ of~\eqref{eq:h-def} is determined by the indicator function, which in turn is determined by the DN map~$\Lambda_D$.
The convex hull of~$\bar D$ can be found in terms of~$h_D$:
\begin{equation}
\operatorname{ch}(\bar D)
=
\bigcap_{\mathclap{\rho\in\sphere^{n-1}}}\{x\in\R^n;x\cdot\rho\leq h_D(\rho)\}.
\end{equation}
\end{proposition}

Moreover, to recover the non-convex part of the superconductive obstacle, one can also use the complex geometrical optic solutions with the logarithmic phase for the Laplacian.
CGOs with logarithmic phase have been used to reconstruct the inclusions with zero or finite conductivities for the Helmholtz equation~\cite{Nakamura:Yoshida:2007,Sini:Yoshida:2012} and for elastic model~\cite{Kar:Sini:2014}.
In the present situation one can also justify the asymptotic properties of the indicator function~\eqref{main_behavior_1far}, \eqref{main_behavior_2far} and~\eqref{main_behavior_3far} using proposition~\ref{proENclo} and the properties of the CGO solutions with logarithmic phase.

\subsection{The probe method for $p=2$}

In this subsection we start with the probe method, which we describe briefly here.
For more details, consult the papers of Ikehata~\cite{Ikehata:1999:jan,Ikehata:1998,Ikehata:2005} and references therein.

Let us first define a ``needle'', which is a continuous map $\gamma\colon [0,1]\to \bar{\Omega}$ such that
$\gamma(0), \gamma(1) \in\doo\Omega$ and $\gamma(t)\in\Omega$ for $0<t<1$.
We also require the trace of the needle to have $n$-dimensional Lebesgue measure zero.

We now introduce the ``impact parameter''~$t(\gamma;D)$ which depends on a needle~$\gamma$ and the obstacle~$D$.
The impact parameter tells us the point in time when a needle hits the obstacle.
We define the hitting time as
\begin{equation}
t(\gamma; D) = \sup \{\tau\in(0,1) ; \gamma(t)\in \Omega\setminus\ol{D} \text{ for all } t<\tau\}.
\end{equation}
One of two cases always happens:
\begin{enumerate}
\item[(i)]
If $t(\gamma;D)<1$, then $\gamma(t(\gamma;D))\in \doo D$.
\item[(ii)]
If $t(\gamma;D) = 1$, then $\gamma$ does not touch any point on~$\doo D$.
\end{enumerate}

We define an indicator function which indicates whether or not a needle touches~$\doo D$.
\begin{equation}
I(t,\gamma) \coloneqq \limsup_{k\to\infty} \ip{(\Lambda_D - \Lambda_{\emptyset}) f_k(\cdot, \gamma(t))}{f_k(\cdot, \gamma(t))},
\end{equation}
where the sequence $\{f_k(\cdot, \gamma(t))\}_k$ is chosen according to the following proposition.
We could also define the indicator function as limit inferior without changing anything.
\begin{proposition}[{\cite[Proposition~2]{Ikehata:1999:oct}}]
\label{prop1lin}
Let~$\gamma$ be a needle as described above.
Then for any $t\in (0,1)$ there exists a sequence of functions $ f_k(\cdot, \gamma(t)) \in H^{1/2}(\doo\Omega)$ such that the solution~$v_k$ of
\begin{equation}
\begin{cases}
\Delta v_k = 0 \ \text{in}\ \Omega \\
v = f_k(\cdot, \gamma(t)) \ \text{on}\ \doo\Omega
\end{cases}
\end{equation}
converges to $\Phi(\cdot; \gamma(t))$ in $H_{loc}^{1}(\Omega\setminus\gamma([0,t]))$ as $k\to\infty$, where~$\Phi$ is the fundamental solution for the Laplacian.
Moreover, for any fixed open $\Gamma\subset\partial\Omega$ with non-empty exterior we can assume that each~$f_k(\cdot, \gamma(t))$ vanishes on~$\Gamma$.
\end{proposition}

The proof of the proposition uses the Runge approximation property of~$-\Delta$.
Notice that the sequence $f_k(\cdot, \gamma(t))$ does not depend on~$D$.
See~\cite[Proposition~2]{Ikehata:1999:oct} for more details.

We define a set~$T(\gamma)$ as
\begin{equation}
T(\gamma) \coloneqq \{\tau\in(0,1) ; I(t,\gamma) \text{ exists and is finite for all }t\in(0,\tau)\}.
\end{equation}
The set~$T(\gamma)$ can be calculated from the DN map.
We will show that
\begin{equation}
\label{eq:T-interval}
T(\gamma) = (0, t(\gamma;D)),
\end{equation}
so the impact parameter can be reconstructed using $t(\gamma; D) = \sup T(\gamma)$.
We also have that $t(\gamma; D)<1$ if and only if $\lim_{t\to t(\gamma; D)} I(t,\gamma) = \infty$.

This will allow us to reconstruct the point~$\gamma(t(\gamma;D))$ for any choice of~$\gamma$.
Under the topological assumptions of theorem~\ref{thm:p2-inverse} these points determine the set~$D$.
To justify these statements, we need some inequalities.

\begin{proposition}
\label{prop:p2-indicator-estimates}
For each needle~$\gamma$ and $t\in (0, t(\gamma;D))$ we have the inequalities
\begin{equation}
\int_{D}\abs{\nabla\Phi(x,\gamma(t))}^2\der x
\leq
I(t,\gamma)
\end{equation}
and
\begin{equation}
I(t,\gamma)
\leq
C
\left(
\int_{D}\abs{\nabla\Phi(x,\gamma(t))}^2\der x + \int_{D}\abs{\Phi(x,\gamma(t))}^2\der x
\right)
,
\end{equation}
where~$C$ is a positive constant.

Moreover, if $t$ is such that $\gamma (t) \in \bar D$, then $I(t,\gamma) = \infty$.
\end{proposition}

\begin{proof}
Consider a sequence of functions $v_k \in H^1(\Omega)$ such that
\begin{equation}
\begin{cases}
\Delta v_k = 0 & \text{in } \Omega  \\
 v_k = f_k & \text{on } \doo\Omega.
\end{cases}
\end{equation}

Applying proposition~\ref{prop1lin} and trace theorem we obtain
\begin{equation}
v_k \to \Phi(\cdot, \gamma(t)) \ \text{in}\ H^1(D)
\end{equation}
and
\begin{equation}
v_k \to \Phi(\cdot, \gamma(t)) \ \text{in}\ H^{1/2}(\doo D)
\end{equation}
when $\bar D \subset \Omega \setminus \gamma\sulut{[0,t]}$.
Hence the estimates follow from lemma~\ref{linLem1} and proposition \ref{proENclo}.

Suppose next that $\gamma(t) \in \bar D$.
For every $\eps > 0$ define $D_\eps = \bar D \setminus \joukko{x \in \Omega; \dist\sulut{x, \gamma\sulut{[0,\gamma(t)]}}<\eps}$, which is a compact subset of $\Omega \setminus \gamma\sulut{[0,\gamma(t)]}$.
In particular, since the trace of $\gamma$ is compact, each point in its complement has positive distance to it and hence
\begin{equation}
\bigcup_{\eps > 0} D_\eps = \bar D \setminus \gamma\sulut{[0,t]}.
\end{equation}
Thus, by proposition~\ref{prop1lin} and lemma~\ref{linLem1}, we obtain
\begin{equation}
I(t,\gamma) \geq \int_{D_\eps} \abs{\nabla\Phi(x,\gamma(t))}^2\der x
\end{equation}
for every $\eps>0$.
Letting $\eps\to0$, we obtain
\begin{equation}
\begin{split}
I(t,\gamma)
&\geq
\int_{D \setminus \gamma\sulut{[0,t]}} \abs{\nabla\Phi(x,\gamma(t))}^2\der x 
\\&=
\int_{D} \abs{\nabla\Phi(x,\gamma(t))}^2\der x 
=
\infty.
\end{split}
\end{equation}
The integrals converge as $\eps\to0$ by the monotone convergence theorem, we can ignore the trace of~$\gamma$ since it has zero measure, and the last integral is infinite by Lipschitz regularity of~$D$ and the singularity of the fundamental solution $\Phi(\cdot,\gamma(t))$ at~$\gamma(t)$.
\end{proof}

We are now ready to present a proof of theorem~\ref{thm:p2-inverse}.

\begin{proof}[Proof of theorem~\ref{thm:p2-inverse}]
Let us first show that equation~\eqref{eq:T-interval} indeed holds true for any needle~$\gamma$.
The fact that the indicator function~$I(t,\gamma)$ exists and is finite for $t\in(0,t(\gamma;D))$ follows from proposition~\ref{prop:p2-indicator-estimates}.
This shows that $\sulut{0,t\sulut{\gamma;D}} \subset T(\gamma)$.

If $t(\gamma;D)<1$, then $x_0\coloneqq\gamma(t(\gamma;D))\in\partial D$, so
\begin{equation}
\lim_{t\to t(\gamma;D)}\int_D\abs{\nabla\Phi(x,\gamma(t))}^2\der x=\infty.
\end{equation}
Therefore proposition~\ref{prop:p2-indicator-estimates} gives $I(t(\gamma;D),\gamma)=\infty$.
This implies that $\sulut{0,t\sulut{\gamma;D}} \supset T(\gamma)$, which proves the identity~\eqref{eq:T-interval}.

Using this identity one can see that the indicator function determines the point~$\gamma(t(\gamma;D))$ (see discussion after the identity), and the DN map in turn determines the indicator function.
If $t(\gamma;D)=1$, we know that~$\gamma$ does not meet~$\bar D$.
If $t(\gamma;D)<1$, we recover the point $\gamma(t(\gamma;D))\in\partial D$.
Since this can be done for all needles~$\gamma$, we can recover~$D$.
\end{proof}

\section{The direct problem}
\label{sec:direct}

\subsection{Well-posedness}
\label{sec:direct-well-posed}

Let us now carefully formulate the direct problem and see that it is well-posed.

\begin{lemma}
\label{lma:poincare-general}
Let $\Omega\subset\R^n$ be a bounded open set with a finite number of connected components, such that each connected component is a Sobolev extension domain (with exponent~$p$) and the closures of the components are disjoint.
Let $\Gamma\subset\partial\Omega$ be open and meet all connected components of~$\Omega$.
There exists a constant~$C$ so that any function $u\in W^{1,p}(\Omega)$ satisfying $u|_\Gamma=0$ in the Sobolev sense satisfies
\begin{equation}
\aabs{u}_{L^p(\Omega)}
\leq
C\aabs{\nabla u}_{L^p(\Omega)}.
\end{equation}
\end{lemma}

\begin{proof}
Let~$W^{1,p}_{0'}(\Omega)$ be the closure under the~$W^{1,p}(\Omega)$ norm of the space of smooth functions in~$\bar\Omega$ supported away from~$\Gamma$.
The prime reminds that the zero boundary value in the Sobolev sense is only assumed on~$\Gamma$ which may be a proper subset of~$\partial\Omega$.

Suppose there was no such constant~$C$.
Then there is a sequence of functions $u_k\in W^{1,p}_{0'}(\Omega)$ so that
\begin{equation}
\label{eq:poincare-contradiction}
\aabs{u_k}_{L^p(\Omega)}
\geq
k\aabs{\nabla u_k}_{L^p(\Omega)}
>
0.
\end{equation}
We may normalize this sequence so that $\aabs{u_k}_{L^p(\Omega)}=1$ and $\aabs{\nabla u_k}_{L^p(\Omega)}<1/k$.
By the Rellich--Kondrachov theorem there is a subsequence converging in~$L^p(\Omega)$.
The Rellich-Kondrachov theorem holds in Sobolev extension domains, and a finite union of Sobolev extension domains with positive distance from each other also admits an extension operator.
We denote the subsequence by~$(u_k)$ and the limit function by~$u$.

For any test function $\eta\in C_0^\infty(\Omega)$ we have
\begin{equation}
\int_\Omega u\nabla\eta
=
\lim_{k\to\infty}\int_\Omega u_k\nabla\eta
=
-\lim_{k\to\infty}\int_\Omega \eta\nabla u_k
=
0,
\end{equation}
using~$W^{1,p}$ regularity of each~$u_k$ and the norm bound on~$\nabla u_k$.
Therefore~$u$ is weakly differentiable and its weak gradient is identically zero.

We have in fact $u\in W^{1,p}(\Omega)$ and the convergence $u_k\to u$ happens also in~$W^{1,p}(\Omega)$, so $u\in W^{1,p}_{0'}(\Omega)$.
Since~$u$ has zero weak gradient, it must be constant on each connected component.
The only possible constant value is zero due to the zero boundary value on~$\Gamma$ and the connectedness assumption.
This contradicts the normalization $\norm{u_k}_{L^p(\Omega)} = 1$ for all~$k$.
\end{proof}

\begin{lemma}
\label{lma:poincare}
Let $\omega,\Omega\subset\R^n$ be bounded domains with Lipcshitz boundaries.
Assume that the closure of each connected component of~$\Omega\setminus\bar\omega$ meets~$\partial\Omega$.
(This happens, in particular, if $\bar\omega\subset\Omega$ and~$\Omega\setminus\bar\omega$ is connected.)
There exists a constant~$C$ so that
\begin{equation}
\aabs{u}_{L^p(\Omega\setminus\bar\omega)}
\leq
C\aabs{\nabla u}_{L^p(\Omega\setminus\bar\omega)}
\end{equation}
for all $u\in W^{1,p}(\Omega\setminus\bar\omega)$ that vanish (in the Sobolev sense) on~$\partial\Omega$.
\end{lemma}

\begin{proof}
This is a special case of lemma~\ref{lma:poincare-general}.
In particular, Lipschitz domains are Sobolev extension domains~\cite[theorem 12]{Calderon:1961}.
\end{proof}

\begin{lemma}
\label{lma:p-harmonic-trace-estimate}
Let $\Omega\subset\R^n$ be a bounded domain.
Any $p$-harmonic function $u\in W^{1,p}(\Omega)$ satisfies
\begin{equation}
\aabs{\nabla u}_{L^p(\Omega)}
\leq
\aabs{u|_{\partial\Omega}}_{W^{1,p}(\Omega)/W^{1,p}_0(\Omega)}.
\end{equation}
\end{lemma}

\begin{proof}
We simply observe that
\begin{equation}
\begin{split}
\aabs{u|_{\partial\Omega}}_{W^{1,p}(\Omega)/W^{1,p}_0(\Omega)}
&=
\inf_{v-u\in W^{1,p}_0(\Omega)}\aabs{v}_{W^{1,p}(\Omega)}
\\&\geq
\inf_{v-u\in W^{1,p}_0(\Omega)}\aabs{\nabla v}_{L^p(\Omega)}
\\&=
\aabs{\nabla u}_{L^p(\Omega)},
\end{split}
\end{equation}
since a $p$-harmonic function minimizes, by definition, the~$L^p$ norm of the gradient.
\end{proof}

\begin{theorem}
\label{thm:direct-variation}
Let $\Omega\subset\R^n$ be a bounded domain and let $\sigma\colon\Omega\to[0,\infty]$ be a measurable function.
Denote $D_0=\sigma^{-1}(0)$ and $D_\infty=\sigma^{-1}(\infty)$.
Suppose the following:
\begin{itemize}
\item The sets~$D_0$ and~$D_\infty$ are open.
\item The sets~$\bar D_0$, $\bar D_\infty$ and~$\partial\Omega$ are disjoint.
\item The function~$\log\sigma$ is essentially bounded in $\Omega\setminus (D_0\cup D_\infty)$.
\item The set~$D_0$ has Lipschitz boundary.
\end{itemize}
Fix $p\in(1,\infty)$.
Given any $f\in W^{1,p}(\Omega)$, there is a minimizer $u\in f+W^{1,p}_0(\Omega)$ to the energy
\begin{equation}
E(u)
=
\int_\Omega\sigma\abs{\nabla u}^p \der x.
\end{equation}
The minimal energy is finite and the minimizer is unique modulo functions that have zero gradient outside~$D_0$ but still satisfy the Dirichlet boundary condition on~$\doo \Omega$.
The minimizer satisfies $\dive(\sigma\abs{\nabla u}^{p-2}\nabla u)=0$ in $\Omega\setminus(\bar D_0\cup\bar D_\infty)$ and $\nabla u=0$ in~$D_\infty$ in the weak sense.
\end{theorem}

Note that the essential boundedness of $\log \sigma$ is equivalent to the existence of $c>0$ such that for almost all $x$ we have $1/c < \sigma(x) < c$.

\begin{remark}
The result which states that the minimizer is unique modulo functions that have zero gradient outside~$D_0$ is simplified when the closure of each connected component of $\Omega \setminus D_0$ intersects $\doo \Omega$ and in particular when $\Omega \setminus D_0$ is connected.
Under this assumption the uniqueness holds modulo $W^{1,p}_0(D_0)$.
\end{remark}

The minimizer can be made unique with a small number of adjustments.
\begin{remark}
Suppose~$u$ and~$v$ are any two minimizers of the energy in theorem~\ref{thm:direct-variation}.
Define~$u_0$ to equal~$u$ in all connected components of $\Omega \setminus D_0$ that touch the boundary~$\doo \Omega$, set $u_0 = 0$ in connected components of $\Omega \setminus D_0$ that do not touch the boundary~$\doo \Omega$ and let~$u_0$ be $p$-harmonic in~$D_0$ with Dirichlet boundary values determined by the previous conditions.
Define~$v_0$ in a similar way, but based on~$v$.
Then $u_0 = v_0$ in $\Omega$ and $\nabla\sulut{u-u_0} = \nabla\sulut{v-v_0} = 0$ in $\Omega \setminus D_0$.
\end{remark}

\begin{proof}[Proof of theorem~\ref{thm:direct-variation}]
First of all, the energy~$E(u)$ is finite if and only if $\nabla u=0$ in~$D_\infty$; in fact, for such functions we have
\begin{equation}
E(u) = \int_{\Omega \setminus D_\infty} \sigma \abs{\nabla u}^p \der x.
\end{equation}
(We use the convention $0\cdot\infty = 0$.)
Since~$\bar D_\infty$ is disjoint from~$\partial\Omega$, there are such functions with the prescribed boundary values.
The space $A=\{u\in W^{1,p}(\Omega);\nabla u|_{D_\infty}=0\}$ is a closed subspace of~$W^{1,p}(\Omega)$, and so is $B=\{u\in W^{1,p}_0(\Omega);\nabla u|_{\Omega\setminus D_0}=0\}$.

It is clear that changing the function~$u$ in~$D_0$ does not change~$E(u)$.
Therefore we consider the quotient space
\begin{equation}
S
=
A/B
.
\end{equation}
The energy functional~$E$ is well defined on this quotient space and $E(u)<\infty$ for all $u\in S$.

Since the only thing that matters about~$f$ are its boundary values and the sets~$\doo \Omega$, $\bar D_0$, and~$\bar D_\infty$ are disjoint, we may assume that~$f$ vanishes on~$D_0$ and~$D_\infty$.

We denote the equivalence of $u\in A$ by~$[u]=u+B$.
Let us define $S_f=\{[u]\in S;u-f\in W^{1,p}_0(\Omega)\}$.
Notice that the truth value of $u-f\in W^{1,p}_0(\Omega)$ does not depend on the choice of the representative of the equivalence class $[u]\subset A$, so~$S_f$ is well-defined.
We will also use the space~$S_0$ where the boundary value is assumed to be zero instead of that of~$f$.
We shall show that there is a unique minimizer of~$E$ in~$S_f$.

Let~$U$ be a connected component of~$\Omega\setminus\bar D_0$ so that $U\cap\partial\Omega=\emptyset$.
Then one can shift the values of $u\in S$ in~$U$ by a constant without changing~$u$ as an element of~$S$.
A minimizer must clearly have vanishing gradient in~$U$, so we may assume that the minimizer (and all functions in a minimizing sequence) vanishes in~$U$.
We therefore assume, to the end of simplifying presentation, from now on that there are no such components~$U$.

For any $u\in S$ we pick a preferred representative $\bar u\in W^{1,p}(\Omega)$ by demanding~$\bar u$ to be $p$-harmonic in~$D_0$.
The boundary values of~$\bar u$ on~$\partial D_0$ are determined by~$u$.

From lemma~\ref{lma:p-harmonic-trace-estimate} we obtain
\begin{equation}
\aabs{\nabla\bar u}_{L^p(D_0)}
\leq
\aabs{\bar u|_{\partial D_0}}_{W^{1,p}(D_0)/W^{1,p}_0(D_0)}.
\end{equation}
By continuity of the quotient map and lemma~\ref{lma:poincare} we have
\begin{equation}
\aabs{u|_{\partial D_0}}_{W^{1,p}(\Omega\setminus D_0)/W^{1,p}_0(\Omega\setminus D_0)}
\leq
C\aabs{\nabla u}_{L^p(\Omega\setminus D_0)}
\end{equation}
for all $u\in W^{1,p}_0\sulut{\Omega}$.
Since the boundary norms on~$\partial D_0$ from different sides are comparable --- in fact both are comparable to the Besov norm on~$B^{1-1/p}_{p,p}(\partial D_0)$ --- we have, for all $u\in W^{1,p}_0\sulut{\Omega}$,
\begin{equation}
\label{eq:int-ext-estimate-s0}
\aabs{\nabla\bar u}_{L^p(D_0)}
\leq
C
\aabs{\nabla u}_{L^p(\Omega\setminus D_0)}.
\end{equation}

Let now~$(u_k)$ be a minimizing sequence of~$E$ in~$S_f$.
Using the estimate~\eqref{eq:int-ext-estimate-s0} and $u_k-f\in S_0$, we get
\begin{equation}
\aabs{\nabla(u_k-f)}_{L^p(\Omega\setminus D_0)}
\geq
C\aabs{\nabla(\overline{u_k-f})}_{L^p(D_0)},
\end{equation}
so for some other constant~$C'$ we have
\begin{equation}
\aabs{\nabla(u_k-f)}_{L^p(\Omega\setminus D_0)}
\geq
C'\aabs{\nabla(\overline{u_k-f})}_{L^p(\Omega)}.
\end{equation}
It follows from the assumption on~$\log\sigma$ that $\sigma|_{\Omega\setminus D_0}\geq\lambda$ for some $\lambda>0$.
Therefore
\begin{equation}
\begin{split}
\lambda^{-1/p}E(u_k)^{1/p}
&\geq
\aabs{\nabla u_k}_{L^p(\Omega\setminus D_0)}
\\&\geq
\aabs{\nabla(u_k-f)}_{L^p(\Omega\setminus D_0)}-\aabs{\nabla f}_{L^p(\Omega\setminus D_0)}
\\&\geq
C'\aabs{\nabla(\overline{u_k-f})}_{L^p(D_0)}-\aabs{\nabla f}_{L^p(\Omega\setminus D_0)}
.
\end{split}
\end{equation}
Since~$E(u_k)$ is bounded, this estimate guarantees that also the sequence $(\overline{u_k-f})$ is bounded in $W^{1,p}_0(\Omega)$.
Therefore there is a subsequence (which we denote by the sequence itself) which converges weakly in~$W^{1,p}(\Omega)$.
Let $u_0-f$ be the limit function.

The energy functional~$E(u)$ is just a weighted Dirichlet energy of $u|_{\Omega\setminus(D_0\cup D_\infty)}$ with weight bounded away from zero and infinity.
Such functionals are weakly lower semicontinuous and $\overline{u_k-f}$ converges weakly to $u_0-f$ also in $W^{1,p}(\Omega\setminus(D_0\cup D_\infty))$ (with the functions restricted appropriately).
Therefore
\begin{equation}
E(u_0)
\leq
\lim_{k\to\infty}E(u_k),
\end{equation}
so~$u_0$ indeed minimizes the energy.

Suppose there are two minimizers~$u_0$ and~$v_0$ of~$E$ in~$S_f$.
Since the functions are distinct, they must differ in $\Omega\setminus(D_0\cup D_\infty)$, and so the set $\{x\in\Omega\setminus(D_0\cup D_\infty);\nabla u_0(x)\neq\nabla v_0(x)\}$ must have positive measure.
By strict convexity of $t\mapsto t^p$ this implies $E(\frac12u_0+\frac12v_0)<\frac12E(u_0)+\frac12E(v_0)$.
But this is impossible because~$u_0$ and~$v_0$ minimize the energy, so $u_0=v_0$ as elements in~$S_f$.

The fact that our unique minimizer solves the PDE follows from standard techniques in the calculus of variations.
\end{proof}

We remark that the weak limit of a locally uniformly bounded sequence of $p$-harmonic functions is $p$-harmonic.\footnote{Every locally uniformly bounded family of weak solutions to the $p$-Laplace equation is equicontinuous~\cite[Theorem~6.12]{Heinonen:Kilpelainen:Martio:1993} and thus by the Ascoli-Arzel\`a theorem has a uniformly convergent subsequence.
A locally uniform limit of solutions is still a solution~\cite[Theorem~3.78]{Heinonen:Kilpelainen:Martio:1993}.
A uniform limit is also a weak~$W^{1,p}$ limit.
}
Therefore the minimizer constructed in the proof above is $p$-harmonic in~$D_0$ at least if~$f$ is bounded.
It is also $p$-harmonic (weighted with~$\sigma$) in $\Omega\setminus(D_0\cup D_\infty)$, but it need not be $p$-harmonic in any sense across~$\partial D_0$.
Using $p$-harmonic extensions to~$D_0$ was just a matter of convenience; the values in~$D_0$ are irrelevant for the energy.
All minimizers are $p$-harmonic in~$D_\infty$ since the gradient must vanish identically there.

The variational problem also leads to a PDE in the whole domain~$\Omega$ despite~$\sigma$ being zero or infinite as we shall see next, as we provide a weak formulation for the equation $\dive(\sigma\abs{\nabla u}^{p-2}\nabla u)=0$.

\begin{theorem}
\label{thm:direct-pde}
Let~$\Omega$ and~$\sigma$ be as in theorem~\ref{thm:direct-variation}.
Fix any $f\in W^{1,p}(\Omega)$ and only consider the functions satisfying the constraint $u-f\in W^{1,p}_0(\Omega)$.
Such a function $u\in W^{1,p}(\Omega)$ minimizes the energy functional~$E$ if and only if $\nabla u=0$ in~$D_\infty$ and
\begin{equation}
\label{eq:euler-lagrange}
\int_\Omega\sigma\abs{\nabla u}^{p-2}\nabla u\cdot\nabla\phi=0
\end{equation}
for all $\phi\in W^{1,p}_0(\Omega)$ satisfying $\nabla\phi=0$ on~$D_\infty$.
\end{theorem}

\begin{proof}
Let us again denote $A=\{u\in W^{1,p}(\Omega);\nabla u|_{D_\infty}=0\}$ and in addition $A_0=A\cap W^{1,p}_0(\Omega)$.
Suppose~$u$ is a minimizer of~$E$ in $\tilde f+A_0$, where $f-\tilde f \in W^{1,p}_0\sulut{\Omega}$ and $\nabla \tilde f = 0$ in~$D_\infty$.
Take any $\phi\in A_0$.
Now
\begin{equation}
\begin{split}
0
&=
\left.\Der{t}E(u+t\phi)\right\rvert_{t=0}
\\&=
\left.\Der{t}\int_\Omega\sigma\abs{\nabla(u+t\phi)}^p\right\rvert_{t=0}
\\&=
\int_\Omega\sigma\left.\Der{t}\abs{\nabla(u+t\phi)}^p\right\rvert_{t=0}
\\&=
p\int_\Omega\sigma\abs{\nabla u}^{p-2}\nabla u\cdot\nabla\phi.
\end{split}
\end{equation}
Commuting differentiation and integration is possible by the dominated convergence theorem and the mean value theorem for the map $t\mapsto\abs{\nabla(u+t\phi)}^p$.

Conversely, suppose~\eqref{eq:euler-lagrange} holds for all $\phi\in A_0$.
The map $\R^n\ni\xi\mapsto\abs{\xi}^p\in\R$ is convex and smooth outside the origin, so
\begin{equation}
\abs{\nabla u(x)}^p+p\abs{\nabla u(x)}^{p-2}\nabla u(x)\cdot\nabla\phi(x)
\leq
\abs{\nabla (u+\phi)(x)}^p
\end{equation}
for all $x\in\Omega$.
Integrating this with weight~$\sigma$ and using~\eqref{eq:euler-lagrange}, we have $E(u)\leq E(u+\phi)$ for all $\phi\in A_0$.
\end{proof}

We also have a result for general~$\phi$ -- not just $\phi\in A_0$ -- by formal integration by parts.
\begin{remark} \label{remark:EL-boundary}
The minimizer of theorem~\ref{thm:direct-pde} also satisfies\footnote{The unit normal~$\nu$ is pointing outward from the perspective of $\Omega\setminus D_\infty$. This means towards the interior of~$D_\infty$.}
\begin{equation}
\label{eq:euler-lagrange-boundary}
\int_\Omega\sigma\abs{\nabla u}^{p-2}\nabla u\cdot\nabla\phi
=
\int_{\partial D_\infty}\sigma\abs{\nabla u}^{p-2}(\partial_\nu u)\phi
\end{equation}
for any $\phi\in W^{1,p}_0(\Omega)$,
provided that the integral over the boundary is well-defined.
This is true, for example, when $\sigma \in C\sulut{\Omega \setminus \sulut{D_0 \cup D_\infty}}$.
\end{remark}

\begin{remark}
\label{rmk:weak-pde-unique}
Given the boundary values, a minimizer~$u$ of the energy~$E$ is not unique, but the function $\sigma\abs{\nabla u}^{p-2}\nabla u$ is unique.
This is the only quantity appearing in the weak Euler--Lagrange equation~\eqref{eq:euler-lagrange}.
In fact, this function is in $L^{p'}(\Omega;\R^n)$ because $\nabla u=0$ whenever $\sigma=\infty$, with the conjugate exponent $p'$ defined by
\begin{equation}
\frac{1}{p} + \frac{1}{p'} = 1.
\end{equation}
\end{remark}

\begin{remark}
\label{rmk:pde-bdy}
We can also formulate the partial differential equation $\dive(\sigma\abs{\nabla u}^{p-2}\nabla u)=0$ (or in weak form~\eqref{eq:euler-lagrange}) in the domain $\Omega\setminus(\bar D_0\cup\bar D_\infty)$.
We only need to find the correct boundary conditions on~$\partial D_0$ and~$\partial D_\infty$.

Let~$C$ be a connected component of~$D_\infty$.
Since~$\nabla u$ must vanish in~$C$,~$u$ is constant on~$C$, but there is also another condition.
We can choose a function $\phi\in C^\infty_0$ so that $\phi\equiv1$ on~$C$ and $\phi\equiv0$ in $D_\infty\setminus C$.
Comparing the two equations in theorem~\ref{thm:direct-pde} (or integrating by parts), we observe that $\int_{\partial C}\sigma\abs{\nabla u}^{p-2}\partial_\nu u=0$.

Let us then find the boundary conditions on~$\partial D_0$.
To that end, we take an arbitrary test function $\phi\in C^\infty_0(\Omega)$ vanishing in~$D_\infty$ for weak Euler--Lagrange equation~\eqref{eq:euler-lagrange}.
Since~$\sigma$ vanishes in~$D_0$, the integral over~$\Omega$ in~\eqref{eq:euler-lagrange} is in fact an integral over~$\Omega\setminus D_0$.
Integration by parts gives $\int_{\partial D_0}\sigma\abs{\nabla u}^{p-2}(\partial_\nu u)\phi=0$.
If this is to hold for all such~$\phi$, we obtain the Neumann boundary condition $\sigma\abs{\nabla u}^{p-2}\partial_\nu u=0$ on~$\partial D_0$.

Therefore the equation of theorem~\ref{thm:direct-pde} can be reformulated as
\begin{equation}
\label{eq:traditional_direct}
\begin{cases}
\dive(\sigma\abs{\nabla u}^{p-2}\nabla u)=0 & \text{ in }\Omega\setminus(\bar D_0\cup\bar D_\infty)\\
u=f & \text{ on }\partial\Omega\\
\sigma\abs{\nabla u}^{p-2}\partial_\nu u=0 & \text{ on }\partial D_0\\
\text{for each component~$C$ of~$D_\infty$}
&
\begin{cases}
u|_C = \text{constant}\\
\int_{\partial C}\sigma\abs{\nabla u}^{p-2}\partial_\nu u=0
.
\end{cases}
\end{cases}
\end{equation}
The constant may be different for different components and the constant values depend on the boundary data~$f$.
\end{remark}

More well-posedness results for the direct problem when $p=2$ can be found in lemmata~\ref{lma:H1-estimates} and~\ref{lma:continuous-constant}.

The next lemma shows that the minimal energy depends monotonically on the conductivity~$\sigma$.
We remind the reader that the functions~$u_\sigma$ and~$u_\gamma$ in the lemma are not unique but the minimal energy is.

\begin{lemma}
\label{lemma:energy_inequality}
Suppose~$\sigma$ and~$\gamma$ are conductivities satisfying the assumptions of theorem~\ref{thm:direct-variation} with $\sigma \leq \gamma$ almost everywhere.
Fix some $f\in W^{1,p}(\Omega)/W^{1,p}_0(\Omega)$ and let~$u_\sigma$ and~$u_\gamma$ solve the $p$-conductivity equation in the sense of theorem~\ref{thm:direct-variation} with the boundary values~$f$ and conductivities~$\sigma$ and~$\gamma$, respectively.
Then
\begin{equation}
\int_\Omega \sigma(x) \abs{\nabla u_\sigma}^p \der x
\leq
\int_\Omega \gamma(x) \abs{\nabla u_\gamma}^p \der x.
\end{equation}
\end{lemma}

\begin{proof}
We define $E_\gamma \colon W^{1,p}(\Omega) \to [0,\infty]$ by
\begin{equation}
E_\gamma(v) = \int_\Omega \gamma(x) \abs{\nabla v}^p \der x
\end{equation}
and~$E_\sigma$ similarly.
For every $v \in f+W^{1,p}_0 (\Omega)$ we have $E_\sigma(v) \leq E_\gamma(v)$.
Taking the infimimum over $v\in f+W^{1,p}_0 (\Omega)$ gives $E_\sigma(u_\sigma) \leq E_\gamma(u_\gamma)$ as desired.
The minimizers exist by theorem~\ref{thm:direct-variation}.
\end{proof}

\subsection{The weak Dirichlet-to-Neumann map}
\label{sec:dn-map}

We can now define the weak Dirichlet-to-Neumann map (DN map).
We assume~$\Omega$ and~$\sigma$ to be as in theorems~\ref{thm:direct-variation} and~\ref{thm:direct-pde} above.
The simpler case where $D_0=D_\infty=\emptyset$ was treated in~\cite{Salo:Zhong:2012,Hauer:2015}.
Let $X=W^{1,p}(\Omega)/W^{1,p}_0(\Omega)$ and~$X'$ be its dual.
The DN map $\Lambda_\sigma\colon X\to X'$ is defined by
\begin{equation}
\ip{\Lambda_\sigma f}{g}
=
\int_\Omega\sigma\abs{\nabla\bar f}^{p-2}\nabla\bar f\cdot\nabla\bar g
,
\end{equation}
where $\bar f\in W^{1,p}\sulut{\Omega}$ is any minimizer of the energy functional~$E$ with boundary values $f\in X$ and $\bar g\in W^{1,p}(\Omega)$ is an extension of $g\in X$ with $\nabla \bar g = 0$ in $D_\infty$.
Since $\bar D_\infty\cap\partial\Omega=\emptyset$, there always exists such an extension~$\bar g$, and the existence of a minimizer~$\bar f$ follows from theorem~\ref{thm:direct-variation}.

Under the same assumptions we also have
\begin{equation}
\ip{\Lambda_\sigma f}{g}
=
\int_{\Omega\setminus D_\infty}\sigma\abs{\nabla\bar f}^{p-2}\nabla\bar f\cdot\nabla\bar g
.
\end{equation}

For a sufficiently nice conductivity~$\sigma$ we can use arbitrary extensions $\bar g \in W^{1,p}(\Omega)$; their gradient does not have to vanish in~$D_\infty$.
See remark~\ref{remark:EL-boundary} for more details.
The DN map $\Lambda_\sigma\colon X\to X'$ has an additional term in such cases:
\begin{equation}
\ip{\Lambda_\sigma f}{g}
=
\int_\Omega\sigma\abs{\nabla\bar f}^{p-2}\nabla\bar f\cdot\nabla\bar g
-
\int_{\partial D_\infty}\sigma\abs{\nabla\bar f}^{p-2}(\partial_\nu\bar f)\bar g
,
\end{equation}
where $\bar f\in W^{1,p}_0\sulut{\Omega}$ is any minimizer of the energy functional~$E$ with boundary values $f\in X$ and $\bar g\in W^{1,p}(\Omega)$ is any extension of $g\in X$.

Let us see why the DN map is well-defined.
As pointed out in remark~\ref{rmk:weak-pde-unique}, the minimizer~$\bar f$ is not unique, but $\sigma\abs{\nabla\bar f}^{p-2}\nabla\bar f$ is.
To see that the definition is independent of the choice of~$\bar g$, we need to show that the above expression for $\ip{\Lambda_\sigma f}{g}$ vanishes when $\bar g\in W^{1,p}_0(\Omega)$ and $\nabla \bar g = 0$ in $D_\infty$ or the extra integral over $\doo D_\infty$ is present.
But this is just the claim of theorem~\ref{thm:direct-pde} and remark~\ref{remark:EL-boundary}.

Linearity of~$\Lambda_\sigma f$ as a functional on~$X$ is evident (although $f\mapsto\Lambda_\sigma f$ is linear if and only if $p=2$).
To see that~$\Lambda_\sigma f$ is continuous 
consider the  estimate
\begin{equation}
\begin{split}
\abs{\int_\Omega\sigma\abs{\nabla\bar f}^{p-2}\nabla\bar f\cdot\nabla\bar g}
\leq
C\aabs{\nabla\bar f}_{L^p(\Omega)}^{p-1}\aabs{\bar g}_{W^{1,p}(\Omega)},
\end{split}
\end{equation}
which gives
\begin{equation}
\aabs{\Lambda_\sigma f}_{X'}
\leq
C\aabs{\nabla\bar f}_{L^p(\Omega)}^{p-1}.
\end{equation}
The constant~$C$ depends on~$\sigma$ but not on~$f$.

In the case $p=2$ we also have the following ellipticity estimate for the DN map.

\begin{proposition}
\label{prop:dn-elliptic}
For any $f\in H^{1/2}(\partial\Omega)=W^{1,2}(\Omega)/W^{1,2}_0(\Omega)$, let $\bar f\in W^{1,2}(\Omega)$ be the extension to~$\Omega$ as a solution of the problem of theorems~\ref{thm:direct-variation} and~\ref{thm:direct-pde} with $p=2$, and let $\bar g \in W^{1,2}(\Omega)$ be any extension of $g$ with $\nabla \bar g = 0$ in $D_\infty$.
To make the extension unique, we demand~$\bar f$ to be harmonic in~$D_0$.
Then the weak DN map is given by
\begin{equation}
\label{eq:dn-no-bdy}
\ip{\Lambda_\sigma f}{g}
=
\int_\Omega\sigma\nabla\bar f\cdot\nabla\bar g
=
\int_{\Omega\setminus \sulut{D_0 \cup D_\infty}}\sigma\nabla\bar f\cdot\nabla\bar g
\end{equation}
and satisfies
\begin{equation}
\label{eq:dn-properties}
\ip{\Lambda_\sigma f}{g}
=
\ip{\Lambda_\sigma g}{f}
=
\ip{\Lambda_\sigma g}{f+c}
\end{equation}
for any constant~$c$.
There is a constant $\lambda>0$ so that
\begin{equation}
\label{eq:dn-elliptic}
\inf_{c\in\R}
\lambda\aabs{f+c}_{H^{1/2}(\partial\Omega)}^2
\leq
\ip{\Lambda_\sigma f}{ f}
\leq
\inf_{c\in\R}
\lambda^{-1}\aabs{f+c}_{H^{1/2}(\partial\Omega)}^2.
\end{equation}
\end{proposition}

\begin{proof}
Since~$\bar f$ solves the problem, it is constant in all connected components of~$D_\infty$.
The properties of the DN map given in~\eqref{eq:dn-properties} are easy to see.

Using~\eqref{eq:dn-properties} we observe that $\ip{\Lambda_\sigma f}{f}=\ip{\Lambda_\sigma f+c}{f+c}$ for any $c\in\R$.
If~$V$ denotes the space of constant functions on~$\partial\Omega$, the DN map is a well-defined bilinear functional on the quotient Hilbert space $H^{1/2}(\partial\Omega)/V$.
If the equivalence class of $f\in H^{1/2}(\partial\Omega)$ in~$H^{1/2}(\partial\Omega)$ is denoted by~$[f]$, we have
\begin{equation}
\aabs{[f]}_{H^{1/2}(\partial\Omega)/V}
=
\inf_{c\in\R}
\aabs{f+c}_{H^{1/2}(\partial\Omega)}.
\end{equation}
Continuity of this bilinear map follows from~\eqref{eq:dn-no-bdy}, and this establishes the upper bound in~\eqref{eq:dn-elliptic}.

Suppose the lower bound did not hold.
Then there is a sequence of functions $f_k\in H^{1/2}(\partial\Omega)$ so that $\aabs{[f_k]}_{H^{1/2}(\partial\Omega)/V}=1$ and $\ip{\Lambda_\sigma [f_k]}{[f_k]}\to0$ as $k\to\infty$.
We can shift the functions~$f_k$ without changing their equivalence classes so that $\aabs{[f_k]}_{H^{1/2}(\partial\Omega)/V}=\aabs{f_k}_{H^{1/2}(\partial\Omega)}$ for all~$k$.

Since $\aabs{f_k}_{H^{1/2}(\partial\Omega)}=1$ for all~$k$, the sequence~$(\bar f_k)$ is bounded in~$W^{1,2}(\Omega)$.
Therefore the sequence~$(f_k)_k$ has a subsequence (denoted by~$(f_k)_k$ itself) so that~$\bar f_k$ converges weakly to some function~$\bar f$ in~$W^{1,2}(\Omega)$, $f\in H^{1/2}(\partial\Omega)$.

Since the Dirichlet energy is weakly lower semicontinuous, we have $\ip{\Lambda_\sigma [f_k]}{[f_k]}\to0$.
This implies that $\ip{\Lambda_\sigma [f]}{[f]}=0$ and also that $\aabs{\nabla\bar f_k}_{L^2(\Omega\setminus D_0)}\to0$.
The latter observation combined with~\eqref{eq:int-ext-estimate-s0} shows that in fact $\aabs{\nabla\bar f_k}_{L^2(\Omega)}\to0$.
Therefore~$\bar f_k$ converges to a constant in the strong sense, and~$[f]=0$.
But
\begin{equation}
\aabs{[f]}_{H^{1/2}(\partial\Omega)/V}
=
\lim_{k\to\infty}\aabs{[f_k]}_{H^{1/2}(\partial\Omega)/V}
=
1.
\end{equation}
This contradiction finishes the proof of the lower bound in~\eqref{eq:dn-elliptic}.
\end{proof}

If we have more control on the boundary values, the ellipticity result becomes nicer.
The proof is similar to the previous one, so we give an abridged version.

\begin{proposition}
\label{prop:dn-elliptic2}
Let $\Gamma\subset\partial\Omega$ be open and meet all connected components of~$\Omega$.
For any $f\in H^{1/2}(\partial\Omega)=W^{1,2}(\Omega)/W^{1,2}_0(\Omega)$, let $\bar f\in W^{1,2}(\Omega)$ be as in proposition~\ref{prop:dn-elliptic}.
There is a constant $\lambda>0$ so that
\begin{equation}
\label{eq:dn-elliptic2}
\lambda\aabs{f}_{H^{1/2}(\partial\Omega)}^2
\leq
\ip{\Lambda_\sigma\bar f}{\bar f}
\leq
\lambda^{-1}\aabs{f}_{H^{1/2}(\partial\Omega)}^2
\end{equation}
for all~$f$ that vanish on~$\Gamma$.
\end{proposition}

\begin{proof}
If $A=\{u\in W^{1,2}(\Omega);u|_\Gamma=0\}$ (the closure of the space of smooth functions supported away from~$\Gamma$) and~$B$ is the trace space of~$A$, then we only consider boundary values $f\in B$.
Notice that by lemma~\ref{lma:poincare-general} the norms~$\aabs{\cdot}_{H^1(\Omega)}$ and~$\aabs{\nabla\cdot}_{L^2(\Omega)}$ are comparable on~$A$ and that $\bar f\in A$ for any $f\in B$.

The upper bound follows from continuity of the DN map as a bilinear form, so we focus on the lower bound.
If the lower bound does not hold, there is a sequence of functions $f_k\in H^{1/2}(\partial\Omega)$ so that $\aabs{f_k}_{H^{1/2}(\partial\Omega)}$ is bounded away from zero but $\ip{\Lambda_\sigma\bar f_k}{\bar f_k}\to0$.
The sequence~$(\bar f_k)_k$ is bounded in~$W^{1,2}(\Omega)$, so up to a subsequence it converges weakly to a limit $\bar f\in A$.

By weak lower semicontinuity of the Dirichlet energy $\ip{\Lambda_\sigma\bar f_k}{\bar f_k}\to0$.
Using~\eqref{eq:int-ext-estimate-s0} and~\eqref{eq:dn-no-bdy} we conclude that $\aabs{\nabla\bar f_k}_{L^2(\Omega)}\to0$ and thus $\aabs{\bar f_k}_{H^1(\Omega)}\to0$.
Thus $\bar f_k\to0$ strongly and thus $\bar f=0$ but $\aabs{\bar f}_{H^{1/2}(\partial\Omega)}$ cannot be zero.
This is a contradiction.
\end{proof}

\section{Enclosure method for $p$-Calder\'on problem}
\label{sec:p}

For any subset $D \subseteq \R^n$ we define the (convex) support function $h_D \colon \sphere^{n-1} \to \R$ as
\begin{equation}
h_D(\rho) = \sup_{x \in D} x \cdot \rho.
\end{equation}

We will use specific solutions that are oscillating in one direction and have exponential behaviour in a perpendicular direction.
The solutions were first introduced by Wolff~\cite[section 3]{Wolff:2007} (see also~\cite{Lewis:1988}) and later applied to inverse problems by Salo and Zhong~\cite[section 3]{Salo:Zhong:2012}.
\begin{definition}[Wolff solutions] \label{def:wolff}
For directions $\rho,\rho^\perp \in \R^n$, parameters $t \in \R,\tau>0 $, and for points $x \in \R^n$
we define the functions
\begin{equation}
u(x,\tau,t,\rho,\rho^\perp) = \exp\sulut{\tau\sulut{x \cdot \rho - t}} w\sulut{\tau x \cdot \rho^\perp},
\end{equation}
where~$w$ is defined in lemma~\ref{lemma:wolff}.
When $\rho, \rho^{\perp} \in \R^n$ satisfy $\abs{\rho} = \abs{\rho^{\perp}} = 1$ and $\rho \cdot \rho^{\perp} = 0$, we call them the Wolff solutions to the $p$-Laplace equation.
We also write $f = u|_{\doo \Omega}$.
\end{definition}

The solutions are $p$-harmonic:
\begin{lemma}\label{lemma:wolff}
Let $\rho, \rho^{\perp} \in \R^n$ satisfy $\abs{\rho} = \abs{\rho^{\perp}} = 1$ and $\rho \cdot \rho^{\perp} = 0$. Define $h \colon \R^n \to \R$ by $h(x) = e^{-\rho \cdot x}w(\rho^\perp \cdot x)$, where the function~$w$ solves the differential equation
\begin{equation}\label{eq:wolff}
w''(s) + V(w,w')w = 0
\end{equation}
with
\begin{equation}
V(w,w') = \frac{(2p-3)\left(w'\right)^2+(p-1)w^2}{(p-1)\left(w'\right)^2 + w^2},
\end{equation}
The function~$h$ is then $p$-harmonic.

Given any initial conditions $(a_0,b_0) \in \R^2 \setminus \{(0,0)\}$ there exists a solution $w \in C^\infty(\R)$ to the differential equation~\eqref{eq:wolff} which is periodic with period $\lambda_p > 0$, satisfies the initial conditions $(w(0),w'(0)) = (a_0,b_0)$, satisfies $\int_0^{\lambda_p} w(s) \der s = 0$, and furthermore there exist constants~$c$ and~$C$ depending on $a_0,b_0,p$ such that for all $s \in \R$ we have
\begin{equation}\label{eq:wolff_new}
C > w(s)^2+w'(s)^2 > c > 0.
\end{equation}
\end{lemma}
For proof see~\cite[Lemma~3.1]{Salo:Zhong:2012} and~\cite[Lemma~3.1]{Brander:Kar:Salo:2015}.
In particular the gradient of the Wolff solutions is
\begin{equation}
\label{eq:nabla_wolff}
\nabla u = \tau \expo\sulut{\tau(x \cdot \rho - t)}\sulut{\rho w\sulut{\tau x \cdot \rho^\perp} + \rho^\perp w'\sulut{\tau x \cdot \rho^\perp}}.
\end{equation}

In this section we assume that the conductivity $\sigma$ is constant~$1$ outside the possibly empty open Lipschitz sets $D_0 = \sigma^{-1}\sulut{\joukko{0}}$ and $D_\infty = \sigma^{-1}\sulut{\joukko{\infty}}$.
Note that this conductivity function satisfies the assumptions of theorem~\ref{thm:direct-variation}, so the forward problem is well-posed.
For definition of the DN~map~$\Lambda_\sigma$, see section~\ref{sec:dn-map}.
\begin{definition}[Indicator function]\label{def:indicator}
Let~$f$ denote the Wolff solutions defined in definition~\ref{def:wolff}.
Then we define the indicator function~$I$ by
\begin{equation}
I\sulut{t,\tau,\rho,\rho^\perp} = \tau^{n-p}\left\langle \sulut{\Lambda_\sigma - \Lambda_\emptyset }f,f \right\rangle,
\end{equation}
where~$\Lambda_\emptyset$ is the DN map associated with no obstacle and conductivity~$1$.
We use the shorthand notation $(\Lambda_\sigma-\Lambda_\emptyset)f\coloneqq\Lambda_\sigma f-\Lambda_\emptyset f$ although the DN maps are non-linear.
As we keep the directions~$\rho$ and~$\rho^\perp$ fixed, we often omit them from our notation and write the indicator function simply as~$I(t,\tau)$.
\end{definition}
We record the following equality, which follows from the definition of the Wolff solutions (definition~\ref{def:wolff}) and the indicator function (definition~\ref{def:indicator}).
\begin{lemma} \label{lemma:time_varies}
$I(t,\tau) = \expo(2\tau(h_D(\rho)-t))I(h_D(\rho), \tau)$.
\end{lemma}
Recall that~$h_D(\rho)$ is the convex support function.

The following lemma is crucial for the proof of the lower bound:
\begin{lemma}\label{lemma:lower_lip}
Suppose $1 < p < \infty$ and $D \subset \Omega$ has Lipschitz boundary.
Then, for sufficiently large $\tau > 0$, we have
\begin{equation}
\int_D \expo\sulut{-p\tau\sulut{h_D(\rho) - x \cdot \rho}} \der x \geq C\tau^{-n}.
\end{equation}
\end{lemma}
For proof, see~\cite[Lemma~4.7]{Brander:Kar:Salo:2015}.

\begin{lemma} \label{lemma:lower_0}
Suppose $D_\infty = \emptyset$ and~$D_0$ has Lipschitz boundary.
Then
\begin{equation}
\abs{I(h_{D_0}(\rho),\tau)} > C >0
\end{equation}
for sufficiently large~$\tau$.
\end{lemma}
\begin{proof}
We write $D = D_0$.

Note that $\doo D_\infty = \emptyset$ and so the indicator function can be rewritten as
\begin{equation}
I(t,\tau) = \tau^{n-p}\sulut{\int_{\Omega\setminus\overline{D}}\abs{\nabla u_z}^p\der x - \int_{\Omega}\abs{\nabla u}^p\der x},
\end{equation}
where~$u_z$ solves the Zaremba problem (see equation~\eqref{eq:traditional_direct})
\begin{equation} \label{eq:zaremba_0}
\begin{cases}
\dive(\abs{\nabla u_z}^{p-2}\nabla u_z) = 0  & \text{ in } \Omega \setminus \ol D \\
u_z = f & \text{ on } \doo \Omega \\ 
\abs{\nabla u_z}^{p-2} \nabla u_z \cdot \nu = 0 & \text{ on } \doo D
\end{cases}
\end{equation}
and~$u$ are the Wolff solutions.



By theorem~\ref{thm:direct-pde} we get
\begin{equation}\label{eq:subt}
\int_{\Omega\setminus\overline{D}}\abs{\nabla u_z}^{p-2}\nabla u_z \cdot \nabla(u_z-u)\der x = 0,
\end{equation}
since $u|_{\doo \Omega} = u_z|_{\doo \Omega} = f$.
For $1<p<\infty$, we now recall the inequality
\begin{equation}
\label{eq:p-norm-convex}
\abs{\eta}^p \geq \abs{\zeta}^p + p\abs{\zeta}^{p-2}\zeta\cdot(\eta-\zeta),
\end{equation}
for all $\zeta, \eta \in\R^n$.
Replacing~$\eta$ by~$\nabla u$ and~$\zeta$ by~$\nabla u_z$ in the above inequality and then integrating over~$\Omega\setminus\overline{D}$, we obtain
\begin{equation}
\begin{split}
\int_{\Omega\setminus\overline{D}}\abs{\nabla u}^p\der x
&\geq
\int_{\Omega\setminus\overline{D}} \abs{\nabla u_z}^p \der x
\\&\quad
+ p\int_{\Omega\setminus\overline{D}}\abs{\nabla u_z}^{p-2}\nabla u_z\cdot \nabla(u-u_z) \der x.
\end{split}
\end{equation}
Finally, using~\eqref{eq:subt} we have
\begin{equation}
\int_{\Omega\setminus\overline{D}}\abs{\nabla u_z}^p\der x \leq \int_{\Omega\setminus\overline{D}} \abs{\nabla u}^p
\der x.
\end{equation}
Therefore, the estimate for the indicator function becomes
\begin{equation}
\begin{split}
I(t, \tau)
& = \tau^{n-p} \sulut{\int_{\Omega\setminus\overline{D}} \abs{\nabla u_z}^p \der x - \int_{\Omega} \abs{\nabla u}^p\der x} \\
& \leq\tau^{n-p} \sulut{\int_{\Omega\setminus\overline{D}} \abs{\nabla u}^p \der x - \int_{\Omega} \abs{\nabla u}^p\der x} \\
& = -\tau^{n-p}\int_{D} \abs{\nabla u}^p \der x,
\end{split}
\end{equation}
that is, 
\begin{equation} 
-I(t, \tau) \geq \tau^{n-p}\int_{D} \abs{\nabla u}^p \der x.
\end{equation}
Hence, combining the gradient of Wolff functions~\eqref{eq:nabla_wolff} and lemma~\ref{lemma:lower_lip} we obtain at $t=h_D(\rho)$
\begin{equation}
\begin{split}
\abs{I(h_D(\rho), \tau)}
& \geq C \tau^{n-p}\int_D \tau^p e^{-p\tau(h_D(\rho)-x \cdot\rho)} \der x \\
& \geq C \int_D \tau^n e^{-p\tau(h_D(\rho)- x\cdot\rho)}\der x \\
& \geq C >0.
\end{split}
\end{equation}
This estimate concludes the proof.
\end{proof}

\begin{lemma}\label{lemma:lower_infty}
Suppose $D_0 = \emptyset$ and~$D_\infty$ has Lipschitz boundary.
Then
\begin{equation}
\abs{I(h_{D_\infty}(\rho),\tau)} > C >0
\end{equation}
for sufficiently large~$\tau$.
\end{lemma}

\begin{proof}
We write $D = D_\infty$.
Recall that
\begin{equation}\label{DN:map:p}
\begin{split}
\ip{\Lambda_Df}{g}
&=
\int_{\Omega\setminus\overline{D}}\abs{\nabla u_{\infty}}^{p-2}\nabla u_{\infty}\cdot\nabla\phi\der x
\end{split}
\end{equation}
where~$u_{\infty}$ satisfies (see equation~\eqref{eq:traditional_direct})
\begin{equation}
\begin{cases}
\dive(\abs{\nabla u_{\infty}}^{p-2}\nabla u_{\infty}) =  0  & \text{ in } \Omega \setminus \ol D \\ 
u_{\infty} = \text{constant}  & \text{ in each component of }  \ D \\ 
\int_{\doo D}\abs{\nabla u_{\infty}}^{p-2} \frac{\doo u_{\infty}}{\doo\nu} =  0 &  \\ 
u_{\infty} =  f  & \text{ on } \doo \Omega \\ 
\end{cases}
\end{equation}
and $\phi\in W^{1,p}(\Omega)$ is an extension of $g \in W^{1,p}(\Omega)/W_{0}^{1,p}(\Omega)$
with $\phi|_{\doo\Omega} = g$ and $\nabla \phi = 0$ in $D_\infty$.
Recall that~$u$ is the Wolff solution for the $p$-Laplacian and $u=f$ on~$\doo\Omega$.
Replacing~$\phi$ in~\eqref{DN:map:p} by~$u_{\infty}$ we obtain
\begin{equation}
\begin{split}
\ip{\Lambda_Df}{f}
& = \int_{\Omega\setminus\overline{D}}\abs{\nabla u_{\infty}}^p\der x.
\end{split}
\end{equation}
On the other hand, for the free DN map we write
\begin{equation}\label{DN:free:p} 
\ip{\Lambda_\emptyset f}{g} = \int_{\Omega}\abs{\nabla u}^{p-2}\nabla u\cdot\nabla\phi \;\der x
\end{equation}
where $\phi \in W^{1,p}(\Omega)$ is an extension of~$g$ with $g=\phi|_{\partial\Omega}\in W^{1,p}(\Omega)/W_{0}^{1,p}(\Omega)$.
Since $u = u_{\infty} = f$ on~$\doo\Omega$, by replacing~$\phi$ in~\eqref{DN:free:p} by~$u$ or~$u_{\infty}$ we get
\begin{equation}
\label{eq:free:p}
\begin{split}
\ip{\Lambda_\emptyset f}{f}
& = \int_{\Omega}\abs{\nabla u}^{p-2}\nabla u\cdot\nabla u_{\infty}\der x \\
& = \int_{\Omega}\abs{\nabla u}^p\der x.
\end{split}
\end{equation}
Therefore,
\begin{equation}
\begin{split}
\ip{(\Lambda_D-\Lambda_\emptyset )f}{f}
& = \int_{\Omega\setminus\overline{D}} \abs{\nabla u_{\infty}}^p \der x - \int_{\Omega}\abs{\nabla u}^p \der x \\
& = (p-1)\int_{D}\abs{\nabla u}^p \der x
\\&\quad
- \left[p\int_{D}\abs{\nabla u}^p \der x+ \int_{\Omega\setminus\overline{D}}\abs{\nabla u}^p \der x- \int_{\Omega\setminus\overline{D}}\abs{\nabla u_{\infty}}^p \der x\right] \\
& = (p-1)\int_{D}\abs{\nabla u}^p \der x - \mathcal{I}.
\end{split}
\end{equation}
Now, by using $\nabla u_{\infty} = 0$ in~$D$ and the identity~\eqref{eq:free:p}, we get
\begin{equation}
\begin{split}
\mathcal{I}
& = p\int_{\Omega}\abs{\nabla u}^p \der x - \int_{\Omega\setminus\overline{D}}\abs{\nabla u_{\infty}}^p \der x + \sulut{1-p} \int_{\Omega\setminus\overline{D}}\abs{\nabla u}^p\der x \\
& = p\int_{\Omega}\abs{\nabla u}^{p-2}\nabla u \cdot \nabla u_{\infty} \der x - \int_{\Omega\setminus\overline{D}}\abs{\nabla u_{\infty}}^p \der x - (p-1)\int_{\Omega\setminus\overline{D}}\abs{\nabla u}^p \der x \\
& = p\int_{\Omega\setminus\overline{D}}\abs{\nabla u}^{p-2}\nabla u \cdot \nabla u_{\infty} \der x - \int_{\Omega\setminus\overline{D}}\abs{\nabla u_{\infty}}^p \der x - (p-1)\int_{\Omega\setminus\overline{D}}\abs{\nabla u}^p \der x.
\end{split}
\end{equation}
Using inequality~\eqref{eq:p-norm-convex} for $\eta=\nabla u_{\infty}$ and $\zeta=\nabla u$,
we obtain that $\mathcal{I}\leq 0$, that is,
\begin{equation}
I(t,\tau) = \tau^{n-p}\ip{(\Lambda_D-\Lambda_\emptyset )f}{f} \geq (p-1)\tau^{n-p}\int_{D}\abs{\nabla u}^p \der x.
\end{equation}
By the properties of the Wolff solutions~\eqref{eq:nabla_wolff} we get
\begin{align}
I(h_D(\rho), \tau) \geq C \tau^{n} \int_{D} \exp\sulut{-p\tau\sulut{h_D(\rho)-x\cdot \rho}} \der x.
\end{align}
Applying lemma~\ref{lemma:lower_lip} finishes the proof.
\end{proof}

\begin{theorem}[Lower bound for the indicator function]
\label{thm:p_lower_bound}
When $t < h_D(\rho)$ and either $D_0 = \emptyset$ or $D_\infty = \emptyset$, and the non-empty~$D$ has Lipschitz boundary, then there exist positive constants~$C_1$ and~$C_2$ such that for sufficiently large~$\tau$ we have
\begin{equation}
\abs{I(t,\tau)} \geq C_1\exp\sulut{C_2\tau}.
\end{equation}

Furthermore, the sign of the indicator function depends on the non-empty~$D$; if~$D_0$ is non-empty, then the indicator function is negative, and if~$D_\infty$ is non-empty, then the indicator function is positive.
\end{theorem}
\begin{proof}
By lemmata~\ref{lemma:lower_infty} and~\ref{lemma:lower_0} the indicator function at time~$h_D$ is bounded from below.
By lemma~\ref{lemma:time_varies} the main part of the theorem holds.

The sign of the indicator function agrees with the sign of
\begin{equation}
\ip{\sulut{\Lambda_D - \Lambda_\emptyset} f}{f}.
\end{equation}
where the $f$ are the Wolff solutions.
That is, we want to prove
\begin{equation}
\ip{\Lambda_D f}{f}
\geq
\ip{\Lambda_\emptyset f}{f}
\end{equation}
when $D=D_\infty$ and the opposite inequality when $D=D_0$.
On the left-hand side, let the extension of~$f$ solve the equation with the inclusion~$D$ and on the right-hand side without it.
The desired estimate then follows from lemma~\ref{lemma:energy_inequality}.
\end{proof}

The enclosure method requires both a lower bound (as provided in theorem~\ref{thm:p_lower_bound}) and an upper bound for the indicator function; see for example the reconstruction procedure in section~\ref{sec:p2-enclosure} and also~\cite[proof of theorem~4.1]{Brander:Harrach:Kar:Salo:2017}.
The lower bound allows us to determine when a half-space intersects the inclusion; or, equivalently, when $t < h_D(\rho)$.
The upper bound would allow us to determine when a half-space does not intersect the inclusion; or, equivalently, when $t > h_D(\rho)$.
Since we only have the lower bound, we only know when we intersect the inclusion, and thus can only find a (convex) superset of the inclusion~$D$.
We have no control over how badly $D'$ in the corollary below overestimates $D$.

\begin{corollary}
\label{cor:general-p}
Suppose that $\Omega \subset \R^n$ is open and bounded with a priori known constant conductivity~$\sigma$ outside an obstacle $D = D_0 \cup D_\infty$.
Suppose $\sigma \colon \Omega \to \R_+ \cup \joukko{0} \cup \joukko{\infty}$ is measurable with $D_0 = \sigma^{-1}\sulut{\joukko{0}}$ and $D_\infty = \sigma^{-1}\sulut{\joukko{\infty}}$.
Suppose either~$D_0$ or~$D_\infty$ is empty, the sets~$\doo \Omega$, $\ol{D_0}$ and~$\ol{D_\infty}$ are pairwise disjoint, the set~$D$ has Lipschitz boundary, and~$\sigma|_{\Omega \setminus \ol{D}}$ is constant.

Then we can find a convex set~$D'$, which is a superset of the convex hull of~$D$.
Furthermore, we can detect whether~$D=\emptyset$ or not.
\end{corollary}

\bibliographystyle{plain}
\bibliography{math}

\end{document}